\documentclass[11pt]{article}

\usepackage[a4paper,margin=1in]{geometry}
\usepackage{amsmath,amssymb,amsfonts,amsthm}
\usepackage{graphicx}
\usepackage{booktabs}
\usepackage{hyperref}
\usepackage{enumitem}

\usepackage{pgfplots}

\usepackage{tikz}
\usetikzlibrary{arrows.meta,positioning,shapes.geometric,calc,patterns}

 \newcommand{\citep}{\cite}
 
\usetikzlibrary{decorations.pathreplacing}

\newtheorem{definition}{Definition}
\newtheorem{assumption}{Assumption}
\newtheorem{proposition}{Proposition}
\newtheorem{theorem}{Theorem}

\def\beq{\begin{equation}}
\def\eeq{\end{equation}}
\def\eq#1{(\ref{#1})}

\def\bar#1{ \begin{array}{#1}}
	\def\ear{\end{array} }

\title{A Delayed Generalized Lotka--Volterra Model for Threshold Instability and Structural Recognition}

\author{Carlo Cattani}

\date{}

\begin{document}

\maketitle

\begin{center}	
	Engineering School, DEIM,  University of Tuscia, (VT), 01100, Italy
	e-mail: cattani@unitus.it	
\end{center}

\begin{abstract}
This paper develops a delayed generalized Lotka--Volterra model for systems in which structural change and recognition of change are not synchronized. The starting point is the idea that a system may accumulate internal transformations while remaining locally stable at the level of interpretation, governance, or perception. Instability becomes effective only when accumulated structural stress crosses a threshold and is recognized after a delay.

We introduce a vector of civilizational state variables including systemic complexity, informational entropy, systemic pressure, resilience, and legitimacy. Their interaction is modeled through a delayed generalized Lotka--Volterra system. The delays represent the temporal separation between structural variation and its cognitive, institutional, or social effects. A composite instability functional is then defined, and threshold crossing is interpreted as the transition from latent instability to recognized instability.

The paper studies positivity, equilibria, linearization, and local stability of the delayed system. Particular attention is given to delay-induced instability: even when the corresponding non-delayed system is locally stable, sufficiently large delays may destabilize the equilibrium and generate oscillatory or nonlinear behavior. The model provides a mathematical framework for delayed recognition, threshold instability, and the emergence of crises in complex adaptive systems.
\end{abstract}



\section{Introduction}

Many complex systems do not become unstable at the same moment in which instability is recognized. Structural deterioration may accumulate gradually, while agents, institutions, or observers continue to interpret the system as locally stable. The recognition of instability may therefore occur only after a delay.

The present work is situated at the intersection of delayed differential equations, generalized Lotka--Volterra dynamics, and threshold models for critical transitions. Delay differential equations provide a natural framework for systems whose present evolution depends on past states; in particular, Rihan~\cite{rihan2021dde} gives a broad treatment of qualitative and numerical methods for DDEs and emphasizes that memory terms may generate dynamics that are not present in the corresponding ordinary differential equations. Within the Lotka--Volterra literature, Saeedian et al.~\cite{saeedian2022delay} show that delays can substantially modify the stability patterns of generalized Lotka--Volterra systems, moving feasible communities from equilibrium regimes to oscillatory or non-point attractor regimes. Chen and Jiang~\cite{chenjiang2021turinghopf} study delayed Lotka--Volterra competition systems and show how nonlocal delays may generate Turing--Hopf bifurcations and multiple spatiotemporal coexistence states. Complementarily, Liu et al.~\cite{liu2023feasibility} analyze feasibility and stability in large Lotka--Volterra systems with structured interactions, clarifying the distinction between the existence of admissible equilibria and their dynamical stability. Finally, the literature on critical transitions and early-warning signals has emphasized that complex systems may lose resilience before a transition becomes macroscopically visible; recent contributions such as Xu et al.~\cite{xu2023nonequilibrium} and Diekert et al.~\cite{Diekert2025EWSDecisions} stress the need to connect detection of instability with interpretable dynamical mechanisms. 

The originality of the present paper lies in combining these strands into a single delayed-threshold framework. Unlike standard delayed Lotka--Volterra models, which usually focus on population abundances, feasibility, persistence, or ecological stability, we introduce a five-variable structural interpretation based on complexity, informational entropy, systemic pressure, resilience, and legitimacy. Unlike standard threshold or early-warning approaches, the threshold is not imposed as an external diagnostic only; it is coupled to a delayed generalized Lotka--Volterra system through the scalar functional
\beq \label{eq:omega-functional}
\Omega(\tau)
=
\alpha C(\tau)
+
\beta H(\tau)
+
\gamma P(\tau)
-
\eta R(\tau)
-
\mu L(\tau) \quad , \qquad (\alpha  , \beta  , \gamma  , \eta  , \mu \geq 0) 
\eeq
where $C$ denotes systemic complexity, $H$ informational entropy, $P$ systemic pressure, $R$ resilience, and $L$ legitimacy. Complexity, entropy, and pressure are treated as destabilizing variables, while resilience and legitimacy are treated as stabilizing variables.

The main contribution is therefore twofold: first, the paper formulates a delayed generalized Lotka--Volterra model in which recognition and structural change are explicitly separated by interaction delays; second, it shows how positivity, equilibria, local stability, delay-induced instability, and threshold crossing can be studied within the same mathematical framework. In this sense, delayed recognition is not treated merely as a qualitative metaphor, but as a dynamical mechanism capable of changing the stability properties of the system.

The central aim of this paper is to translate this idea into a mathematical model based on delayed differential equations. The conceptual motivation is simple: if structural change and recognition of change are not simultaneous, then the interaction terms in a dynamical system should not always be instantaneous. Delays are not merely technical parameters; they encode the time required for one variable to become effective upon another.

We focus on a generalized Lotka--Volterra system with delays:
\beq\label{eq:general-delayed-lv}
\frac{dx_i}{d\tau}
=
x_i(\tau)
\left(
r_i+\sum_{j=1}^{n}a_{ij}x_j(\tau-\delta_{ij})
\right),
\qquad i=1,\ldots,n.
\eeq
Here $x_i(\tau)$ are non-negative state variables, $r_i$ are intrinsic growth or decay rates, $a_{ij}$ are interaction coefficients from $x_j$ to $x_i$, and $\delta_{ij}\geq 0$ are delays with which this interaction becomes effective. When $\delta_{ij}=0$, the effect of $x_j$ on $x_i$ is instantaneous. When $\delta_{ij}>0$, the effect becomes operative only after a delay.
The model is applied to a five-variable system $n=5$:
\beq \label{eq:five-variable-vector}
x(\tau)=
\big(C(\tau),H(\tau),P(\tau),R(\tau),L(\tau)\big),
\eeq

The paper first introduces delayed recognition and the threshold functional, then studies positivity, equilibria, linearization, and delay-induced instability of the delayed generalized Lotka--Volterra system. The final section presents numerical scenarios illustrating the transition from stable adaptation to threshold crossing.


\section{From Delayed Recognition to Delayed Dynamics}
\label{sec:delayed-recognition}

The purpose of this section is to translate the qualitative principle of
delayed recognition into a dynamical model with explicit delays. The guiding
idea is that structural change and recognition of structural change need not
be synchronized. A system may evolve in chronological time, accumulate internal
stress, and nevertheless continue to be interpreted through an older stable
frame. Recognition occurs only when the accumulated transformation becomes
effective at the cognitive, institutional, or collective level.

Let $\tau$ denote chronological time and let $S_t$ denote the interpretive state
of a system at the $t$-th effective update. The index $t$ does not necessarily
coincide with chronological time. Rather, it marks the sequence of effective
interpretive transformations. A local update may satisfy
\begin{equation}
d(S_{t+1},S_t)<\theta \ ,
\label{eq:local-subthreshold}
\end{equation}
where $d$ is an informational distance and $\theta>0$ is a salience threshold.
Condition \eqref{eq:local-subthreshold} means that the system changes, but the
change remains locally absorbed by the existing interpretive frame. However,
after many locally sub-threshold updates, the accumulated distance from an
earlier reference state may satisfy
\begin{equation}
d(S_n,S_0)\geq \theta.
\label{eq:global-threshold}
\end{equation}
Thus, the system may be locally stable and globally transformed.

This distinction is the conceptual bridge from delayed recognition to delayed
differential dynamics. If the consequences of a variable are not immediately
recognized or institutionally effective, then the interaction terms of the model
should not be assumed to be instantaneous. Instead, the present rate of change
of one variable should depend on past values of other variables. This is the
standard mathematical intuition behind delay differential equations, where the
future trajectory of the system depends not only on its present state, but also
on its history \cite{menon2021delayinduced,saeedian2022delay,pigani2022delay}.

\subsection{Chronological time, recognition time, and effective interaction}

Let
\[
x(\tau)=\big(x_1(\tau),\ldots,x_n(\tau)\big)\in\mathbb{R}^n_+
\]
be the vector of structural variables. In an instantaneous dynamical model, the
effect of $x_j$ on $x_i$ at time $\tau$ is represented by $x_j(\tau)$. This
amounts to assuming that the state of $x_j$ becomes immediately effective in the
evolution equation for $x_i$. Such an assumption is often too strong for complex
social, institutional, ecological, or civilizational systems.

In a delayed model, the present variation of $x_i$ depends on the past value
\[
x_j(\tau-\delta_{ij}),
\]
where $\delta_{ij}\geq 0$ is the delay associated with the influence of variable
$j$ on variable $i$. If $\delta_{ij}=0$, the effect is instantaneous. If
$\delta_{ij}>0$, the effect becomes operative only after a chronological delay.

\begin{definition}[Interaction delay]
	Let $x_i$ and $x_j$ be two state variables. The parameter $\delta_{ij}\geq 0$ is
	called an interaction delay if the contribution of $x_j$ to the evolution of
	$x_i$ at time $\tau$ depends on the past value $x_j(\tau-\delta_{ij})$.
\end{definition}

In the present framework, the delay $\delta_{ij}$ may have several
interpretations. It may be a cognitive delay, when agents require time to
recognize the consequences of a change; an institutional delay, when
organizations require time to react; an informational delay, when signals
circulate slowly or are initially noisy; or an adaptive delay, when a system
cannot immediately modify its behavior after a structural perturbation. In all
cases, the delay measures the separation between structural variation and
effective response.

This is why delayed differential equations are mathematically natural for the
present problem. They formalize the fact that the dynamics of the system are
history-dependent. Recent work on delay-induced stability and instability has
emphasized that delays may qualitatively change the stability of equilibria,
because the corresponding characteristic equations become transcendental rather
than polynomial \cite{menon2021delayinduced}. In Lotka--Volterra-type systems,
delays may produce stability switches, oscillations, or Hopf bifurcations
\cite{xu2022delayedlv,lifan2024twodelay,li2022memorydiffusion}.


\subsection{From recognition delay to delayed interaction terms}

The difference between structural accumulation and recognition can be represented
by two time variables: the chronological time $\tau$ and the update index $t$ of the interpretative state of a system $S_t$. Let $\tau^\ast$ be the chronological time at which a structural
threshold is crossed, and let $\Delta>0$ be the additional recognition delay.
Then
\begin{equation}
\Omega(\tau^\ast)\geq \Theta
\quad\Longrightarrow\quad
\text{recognized instability at } \tau^\ast+\Delta.
\label{eq:recognition-delay}
\end{equation}
Here $\Omega$ is the  scalar instability functional \eq{eq:omega-functional} and $\Theta$ is the instability
threshold. Equation \eqref{eq:recognition-delay} says that instability may exist
before it is recognized.

The crucial point is that \eqref{eq:recognition-delay} does not necessarily
produce immediate recognition. The model therefore separates the structural
event
$
\Omega(\tau^\ast)\geq \Theta
$
from the interpretive update
$
S_t^{\mathrm{}}\longrightarrow S_{t+1}^{\mathrm{}},
$
which may occur only at $\tau^\ast+\Delta$.

\begin{figure}[ht]
	\centering
	\begin{tikzpicture}[
	node distance=1.2cm,
	every node/.style={font=\small},
	box/.style={
		rectangle,
		rounded corners,
		draw,
		align=center,
		minimum width=2.7cm,
		minimum height=0.9cm
	},
	arrow/.style={-{Latex[length=2mm]},thick}
	]
	\node[box] (structural) {Structural\\ accumulation};
	\node[box, right=of structural] (threshold) {Threshold crossing\\ $\Omega(\tau^\ast)\geq\Theta$};
	\node[box, right=of threshold] (delay) {Recognition delay\\ $\Delta>0$};
	\node[box, right=of delay] (update) {Interpretive update\\ $S_t\to S_{t+1}$};
	
	\draw[arrow] (structural) -- (threshold);
	\draw[arrow] (threshold) -- node[above] {$\tau^\ast$} (delay);
	\draw[arrow] (delay) -- node[above] {$\tau^\ast+\Delta$} (update);
	
	\end{tikzpicture}
	\caption{Delayed recognition as a separation between structural threshold crossing and interpretive update.Instability may be structurally present before it is collectively recognized.}
	\label{fig:delayed-recognition}
\end{figure}
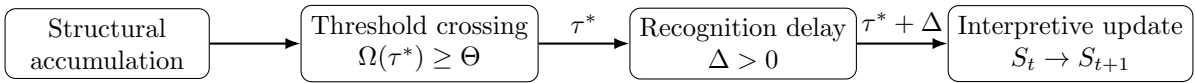

Figure~\ref{fig:delayed-recognition} represents the conceptual structure of the
model. The crossing of the structural threshold occurs at $\tau^\ast$, while the
corresponding interpretive update occurs only after the delay $\Delta$. This
delay may be short in highly responsive systems and long in opaque or rigid
systems. The larger the recognition delay, the longer the system continues to
operate according to an obsolete interpretive frame.


This discussion motivates the introduction of interaction delays inside
the dynamical equations. The general delayed Lotka--Volterra 
Eqs. \eqref{eq:general-delayed-lv} should not be read only as an ecological
population model. In the present paper, it is used as a general nonlinear
interaction scheme. The multiplicative factor $x_i(\tau)$ ensures that the
growth rate of each variable is proportional to its current magnitude, while the
bracketed term represents intrinsic tendencies and delayed interactions with
the other variables. This structure is suitable for variables such as
complexity, pressure, resilience, and legitimacy because these quantities often
change through feedback effects rather than through independent linear
increments.

The passage from instantaneous to delayed dynamics can be expressed as
\begin{equation}
x_j(\tau)
\quad\longmapsto\quad
x_j(\tau-\delta_{ij}).
\label{eq:delay-substitution}
\end{equation}
This substitution is mathematically small but conceptually decisive. It means
that the present evolution of $x_i$ is shaped by a past configuration of the
system. Therefore, the current state is not sufficient to determine the future
trajectory unless the relevant history is also specified.

Accordingly, the initial condition for \eqref{eq:general-delayed-lv} is not a
single vector $x(0)$, but a history function
\begin{equation}
x_i(\tau)=\varphi_i(\tau),
\qquad
\tau\in[-\delta_{\max},0],
\qquad
i=1,\ldots,n,
\label{eq:history-function}
\end{equation}
where
\[
\delta_{\max}=\max_{i,j}\delta_{ij}.
\]
This historical dependence is coherent with the interpretation of delayed
recognition: the present dynamics of the system depend on past states that have
only now become effective.

\subsection{The delayed civilizational interaction scheme}

For the five-variable model, the delayed interaction structure can be written as
\begin{equation}
\begin{cases}
\dot C(\tau)
=
C(\tau)
\left[
r_C
+
a_{CH}H(\tau-\delta_{CH})
-
a_{CR}R(\tau-\delta_{CR})
+
a_{CP}P(\tau-\delta_{CP})
\right],
\\[0.25cm]
\dot H(\tau)
=
H(\tau)
\left[
r_H
+
a_{HC}C(\tau-\delta_{HC})
-
a_{HL}L(\tau-\delta_{HL})
-
a_{HR}R(\tau-\delta_{HR})
\right],
\\[0.25cm]
\dot P(\tau)
=
P(\tau)
\left[
r_P
+
a_{PC}C(\tau-\delta_{PC})
+
a_{PH}H(\tau-\delta_{PH})
-
a_{PR}R(\tau-\delta_{PR})
\right],
\\[0.25cm]
\dot R(\tau)
=
R(\tau)
\left[
r_R
-
a_{RC}C(\tau-\delta_{RC})
-
a_{RH}H(\tau-\delta_{RH})
-
a_{RP}P(\tau-\delta_{RP})
+
a_{RL}L(\tau-\delta_{RL})
\right],
\\[0.25cm]
\dot L(\tau)
=
L(\tau)
\left[
r_L
-
a_{LC}C(\tau-\delta_{LC})
-
a_{LH}H(\tau-\delta_{LH})
-
a_{LP}P(\tau-\delta_{LP})
+
a_{LR}R(\tau-\delta_{LR})
\right].
\end{cases}
\label{eq:five-variable-delayed-system}
\end{equation}

All coefficients $a_{\cdot\cdot}$  are assumed to be non-negative. The signs in \eqref{eq:five-variable-delayed-system} encode the qualitative
assumptions of the model. Complexity, entropy, and pressure tend to amplify
instability, while resilience and legitimacy tend to dampen it. The delays
specify when these effects become operative. For example, $\delta_{HC}$ is the
delay with which complexity increases informational entropy; $\delta_{PH}$ is
the delay with which informational entropy increases systemic pressure;
$\delta_{LP}$ is the delay with which pressure erodes legitimacy; and
$\delta_{RP}$ is the delay with which pressure reduces resilience.  Thus, for example, $a_{HC}C(\tau-\delta_{HC})$ means that complexity increases informational entropy after a delay $\delta_{HC}$, while $-a_{HL}L(\tau-\delta_{HL})$ means that legitimacy reduces entropy after a delay $\delta_{HL}$.

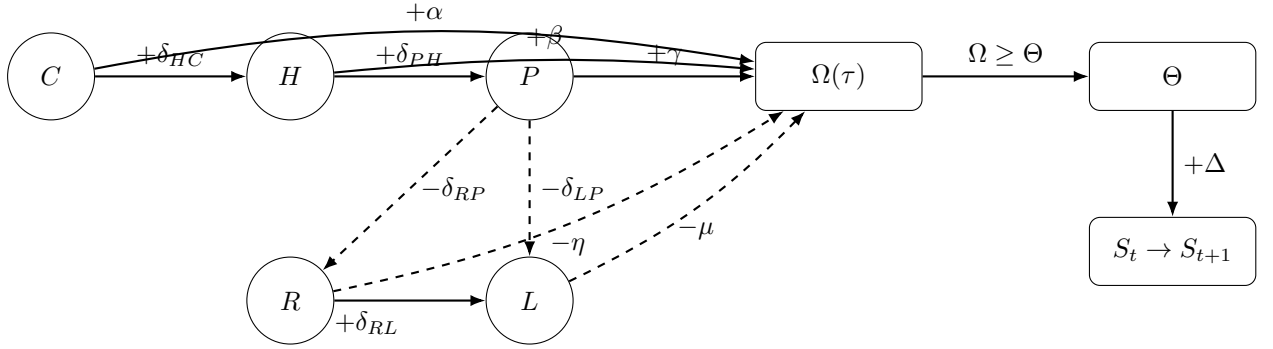
\begin{figure}[ht]
	\centering
	\begin{tikzpicture}[
	node distance=1.8cm,
	every node/.style={font=\small},
	var/.style={
		circle,
		draw,
		align=center,
		minimum size=1.15cm
	},
	func/.style={
		rectangle,
		rounded corners,
		draw,
		align=center,
		minimum width=2.2cm,
		minimum height=0.9cm
	},
	arrow/.style={-{Latex[length=2mm]},thick},
	negarrow/.style={-{Latex[length=2mm]},thick,dashed}
	]
	
	\node[var] (C) {$C$};
	\node[var, right=2.0cm of C] (H) {$H$};
	\node[var, right=2.0cm of H] (P) {$P$};
	
	\node[var, below=1.8cm of H] (R) {$R$};
	\node[var, below=1.8cm of P] (L) {$L$};
	
	\node[func, right=2.4cm of P] (Omega) {$\Omega(\tau)$};
	\node[func, right=2.2cm of Omega] (Theta) {$\Theta$};
	\node[func, below=1.4cm of Theta] (Update) {$S_t\to S_{t+1}$};
	
	\draw[arrow] (C) -- node[above] {$+\delta_{HC}$} (H);
	\draw[arrow] (H) -- node[above] {$+\delta_{PH}$} (P);
	\draw[negarrow] (P) -- node[right] {$-\delta_{RP}$} (R);
	\draw[negarrow] (P) -- node[right] {$-\delta_{LP}$} (L);
	\draw[arrow] (R) -- node[below left] {$+\delta_{RL}$} (L);
	
	\draw[arrow] (C) to[bend left=10] node[above] {$+\alpha$} (Omega);
	\draw[arrow] (H) to[bend left=5] node[above] {$+\beta$} (Omega);
	\draw[arrow] (P) -- node[above] {$+\gamma$} (Omega);
	\draw[negarrow] (R) to[bend right=10] node[below] {$-\eta$} (Omega);
	\draw[negarrow] (L) to[bend right=10] node[below] {$-\mu$} (Omega);
	
	\draw[arrow] (Omega) -- node[above] {$\Omega\geq\Theta$} (Theta);
	\draw[arrow] (Theta) -- node[right] {$+\Delta$} (Update);
	
	\end{tikzpicture}
	\caption{Delayed interaction genealogy for the five-variable instability model. Solid arrows indicate amplifying effects, while dashed arrows indicate stabilizing or inhibiting effects.}
	\label{fig:delayed-genealogy}
\end{figure}

Figure~\ref{fig:delayed-genealogy} summarizes the interaction genealogy. The
diagram should be read dynamically: the arrows do not represent instantaneous
causal links, but delayed effective influences. This is precisely what makes the
model compatible with delayed recognition. The system may continue to accumulate
complexity, entropy, and pressure while resilience and legitimacy are eroded
only after a lag. As a consequence, the functional $\Omega(\tau)$ may approach
or cross the threshold before the system has adapted to the new structural
condition.

\subsection{Mathematical consequence: delay as a destabilizing parameter}

The introduction of delays changes the stability problem. In an instantaneous
system, linearization around an equilibrium generally leads to a polynomial
characteristic equation. In a delayed system, the characteristic equation
contains exponential terms of the form $e^{-\lambda\delta_{ij}}$. This makes the
location of characteristic roots more delicate and allows stability switches as
the delays vary \cite{menon2021delayinduced,lifan2024twodelay}.

In the simplest common-delay case, suppose that all delays are equal:
\[
\delta_{ij}=\delta.
\]
If $x^\ast$ is a positive equilibrium and
\[
B=\operatorname{diag}(x^\ast)A,
\]
then the linearized system has the form
\begin{equation}
\dot u(\tau)=B u(\tau-\delta).
\label{eq:common-delay-linearization}
\end{equation}
Seeking exponential solutions $u(\tau)=ve^{\lambda \tau}$ gives
\begin{equation}
\det\left(\lambda I-Be^{-\lambda\delta}\right)=0.
\label{eq:common-delay-characteristic}
\end{equation}
Equation \eqref{eq:common-delay-characteristic} is transcendental. Therefore,
the stability of the equilibrium cannot be inferred only from the eigenvalues
of $B$ in the same way as in ordinary differential equations. A delay that is
small may preserve stability, while a delay above a critical value may generate
oscillations or instability.

This observation is essential for the present model. Recognition delay is not a
neutral postponement. If the system reacts too late to structural accumulation,
the delay itself becomes part of the instability mechanism. In this sense,
delayed recognition is not merely an epistemic limitation; it is a dynamical
parameter capable of changing the qualitative behavior of the system. The transition from delayed recognition to delayed dynamics can therefore be
summarized as follows:
$
\text{local sub-threshold accumulation}
\quad\Rightarrow\quad
\text{delayed effective interaction}
\quad\Longrightarrow\quad
\text{threshold crossing}
\quad\Longrightarrow\quad
\text{possible delay-induced instability}.
$
The next sections develop this structure by studying positivity, equilibria,
linearization, and local stability of the delayed generalized Lotka--Volterra
system.


\section{A Threshold Functional for Structural Instability}
\label{sec:threshold-functional}

The delayed generalized Lotka--Volterra system \eq{eq:five-variable-delayed-system}  describes the evolution of interacting structural variables. However, the mere evolution of these variables is not yet sufficient to determine whether the system is stable, fragile, or unstable. A further scalar quantity   is needed in order to summarize the balance between destabilizing and stabilizing forces. For this reason, we introduce with \eq{eq:omega-functional} the  threshold functional 
$
\Omega(\tau),
$
which measures the instantaneous structural stress of the system at chronological time $\tau$.

The role of $\Omega$ is analogous to the role of an order parameter in dynamical systems and critical-transition theory. It compresses several interacting dimensions into a single instability index. Such a reduction is not meant to eliminate the multidimensional nature of the system; rather, it provides a mathematically tractable criterion for distinguishing a stable regime from a threshold-crossing regime. Similar scalar indicators are frequently used in the study of resilience, early-warning signals, and critical transitions, where complex systems may undergo abrupt qualitative shifts after gradual changes in underlying variables \citep{lenton2019climate,dakos2024tipping,george2023early,xu2023nonequilibrium}.

Let
$
x(\tau)
$ as given in \eq{eq:five-variable-vector},
denote the civilizational state vector. 
The first three variables $C(\tau) , H(\tau) , P(\tau)  $ are destabilizing: increasing complexity, entropy, and pressure tend to increase the fragility of the system. The last two variables $R(\tau), L(\tau)$ are stabilizing: increasing resilience and legitimacy tend to reduce instability. This motivates the following definition.

\begin{definition}[Structural instability functional]
	\label{def:omega}
	Let
	$
	x(\tau)
	$ be the state vector  \eq{eq:five-variable-vector}
	The structural instability functional $\Omega(\tau)$ is defined by
 \eq{eq:omega-functional}.
	The coefficients $\alpha,\beta,\gamma,\eta,\mu$ measure the relative contribution of each variable to the instability balance.
\end{definition}

The functional \eqref{eq:omega-functional} is intentionally linear at this stage. This choice is not restrictive as a first approximation: it allows the threshold mechanism to be separated from the nonlinear interaction structure of the delayed Lotka--Volterra equations. Nonlinear or state-dependent versions of $\Omega$ could  be also introduced, for instance by including interaction terms such as $C(\tau)H(\tau)$, $P(\tau)/R(\tau)$, or entropy--legitimacy coupling. 

\subsection{Stable, critical, and unstable regimes}

Let
$
\Theta>0
$
be a structural instability threshold. The system is said to be structurally stable when
\begin{equation}
\label{eq:stable-regime}
\Omega(\tau)<\Theta.
\end{equation}
It enters the threshold regime when
\begin{equation}
\label{eq:threshold-regime}
\Omega(\tau)=\Theta,
\end{equation}
and it becomes structurally unstable when
\begin{equation}
\label{eq:unstable-regime}
\Omega(\tau)>\Theta.
\end{equation}

The threshold $\Theta$ should not be interpreted as a universal constant. It depends on the scale of the system, the units used to measure the variables, the normalization adopted, and the institutional or social capacity of the system to absorb perturbations. In the theory of critical transitions, thresholds often separate basins of attraction or mark the loss of resilience of a stable regime \citep{scheffer2009early,dakos2024tipping,rietkerk2025ambiguity}. In the present model, $\Theta$ plays an analogous role: it separates a regime in which structural stress is absorbed from a regime in which structural stress becomes destabilizing.

\begin{definition}[Structural threshold crossing]
	\label{def:threshold-crossing}
	A structural threshold crossing occurs at time $\tau^\ast$ if
	\begin{equation}
	\label{eq:crossing}
	\Omega(\tau^\ast)\geq \Theta
	\end{equation}
	and there exists $\varepsilon>0$ such that
	\begin{equation}
	\label{eq:left-stable}
	\Omega(\tau)<\Theta
	\qquad
	\text{for all } \tau\in(\tau^\ast-\varepsilon,\tau^\ast).
	\end{equation}
	The time $\tau^\ast$ is called the structural crossing time.
\end{definition}

Definition \ref{def:threshold-crossing} distinguishes the moment at which the system objectively crosses the instability threshold from the moment at which the crossing is recognized. This distinction is essential in delayed systems. A system may cross its structural threshold at time $\tau^\ast$, while agents, institutions, or observers recognize the transition only at a later time
\[
\tau^\ast+\Delta,
\]
where $\Delta>0$ is an interpretive or institutional recognition delay.

Thus, the model separates two events:
\[
\Omega(\tau^\ast)\geq \Theta
\qquad
\text{structural instability},
\]
and
\[
\text{recognition at } \tau^\ast+\Delta
\qquad
\text{recognized instability}.
\]
This is the mathematical core of delayed recognition (see Fig. \ref{fig:omega-functional}). Instability may exist before it is perceived.

\begin{figure}[htbp]
	\centering
	\begin{tikzpicture}[
	node distance=1.9cm,
	box/.style={draw, rounded corners, thick, align=center, minimum width=2.3cm, minimum height=0.9cm},
	smallbox/.style={draw, rounded corners, thick, align=center, minimum width=1.7cm, minimum height=0.75cm},
	arrow/.style={-{Latex[length=2.3mm]}, thick},
	dashedarrow/.style={-{Latex[length=2.3mm]}, thick, dashed}
	]
	
	\node[box] (C) {$C(\tau)$\\Complexity};
	\node[box, right=of C] (H) {$H(\tau)$\\Entropy};
	\node[box, right=of H] (P) {$P(\tau)$\\Pressure};
	
	\node[box, below=1.35cm of H] (Omega) {$\Omega(\tau)$\\Instability\\functional};
	
	\node[box, below left=1.15cm and 0.45cm of Omega] (R) {$R(\tau)$\\Resilience};
	\node[box, below right=1.15cm and 0.45cm of Omega] (L) {$L(\tau)$\\Legitimacy};
	
	\node[smallbox, right=2.3cm of Omega] (Theta) {$\Theta$\\Threshold};
	\node[smallbox, right=1.65cm of Theta] (Update) {Recognized\\transition};
	
	\draw[arrow] (C) -- node[above] {$+\alpha$} (Omega);
	\draw[arrow] (H) -- node[right] {$+\beta$} (Omega);
	\draw[arrow] (P) -- node[above] {$+\gamma$} (Omega);
	\draw[arrow] (R) -- node[left] {$-\eta$} (Omega);
	\draw[arrow] (L) -- node[right] {$-\mu$} (Omega);
	
	\draw[arrow] (Omega) -- node[above] {$\Omega\geq\Theta$} (Theta);
	\draw[dashedarrow] (Theta) -- node[above] {$+\Delta$} (Update);
	
	\end{tikzpicture}
	\caption{Construction of the instability functional. Complexity, entropy, and pressure increase structural instability, while resilience and legitimacy reduce it. Threshold crossing may be recognized only after an additional delay $\Delta$.}
	\label{fig:omega-functional}
\end{figure}
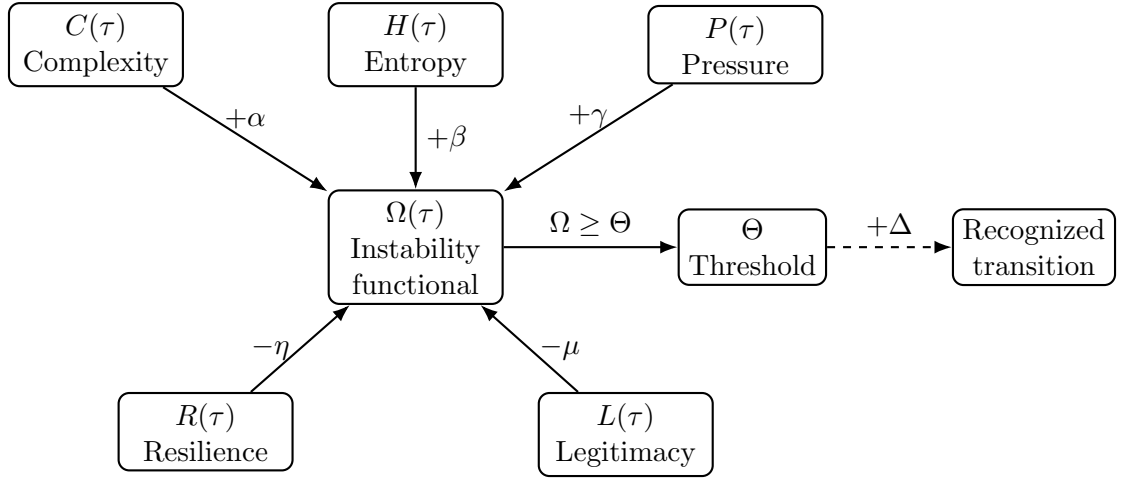 

\subsection{Geometric interpretation}

The threshold functional defines a half-space decomposition of the positive state space. Let
\[
\mathbb{R}_{+}^{5}
=
\left\{
(C,H,P,R,L): C,H,P,R,L\geq 0
\right\}.
\]
The stable region is
\begin{equation}
\label{eq:stable-set}
\mathcal{S}_{\Theta}
=
\left\{
x\in\mathbb{R}_{+}^{5}:
\Omega(x)<\Theta
\right\}.
\end{equation}
The threshold surface is
\begin{equation}
\label{eq:threshold-surface}
\partial\mathcal{S}_{\Theta}
=
\left\{
x\in\mathbb{R}_{+}^{5}:
\Omega(x)=\Theta
\right\},
\end{equation}
and the unstable region is
\begin{equation}
\label{eq:unstable-set}
\mathcal{U}_{\Theta}
=
\left\{
x\in\mathbb{R}_{+}^{5}:
\Omega(x)>\Theta
\right\}.
\end{equation}

Although the full state space is five-dimensional, the mechanism can be represented schematically by projecting the dynamics onto the scalar instability functional $\Omega(\tau)$.

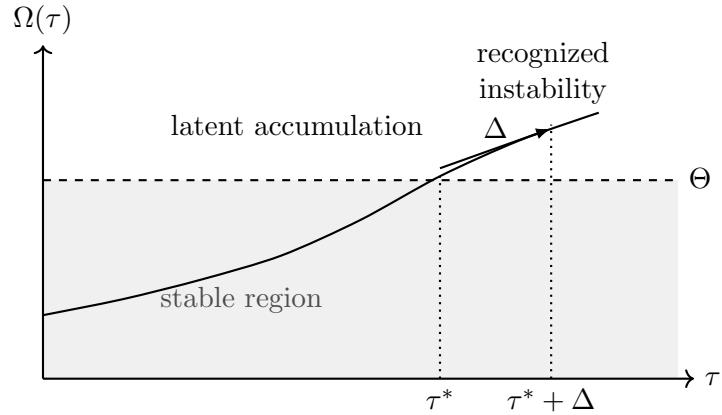
\begin{figure}[ht]
	\centering
	\begin{tikzpicture}[
	scale=1.05,
	axis/.style={->,thick},
	curve/.style={thick},
	thresh/.style={thick,dashed},
	arr/.style={-{Latex[length=2mm]},thick}
	]
	
	
	\fill[gray!12] (0.01,0.01) rectangle (8,2.5);
	\node[gray!60!black] at (2.5,1.0) {stable region};
	
	\draw[thresh] (0,2.5) -- (8,2.5) node[right] {$\Theta_{\mathrm{}}$};
	
	\draw[curve] plot[smooth] coordinates {
		(0,0.8) (1,1.0) (2,1.25) (3,1.55) (4,2.0) (5,2.55) (6,3.0) (7,3.35)
	};

	\draw[->, thick] (0,0) -- (8.2,0) node[right] {$\tau$};
	\draw[->, thick] (0,0) -- (0,4.2) node[above] {$\Omega(\tau)$};

	\draw[dotted,thick] (5,0) -- (5,2.55);
	\draw[dotted,thick] (6.4,0) -- (6.4,3.2);
	
	\node[below] at (5,0.) {$\tau^\ast$};
	\node[below] at (6.4,0) {$\tau^\ast+\Delta$};
	
	\node[align=center] at (3.2,3.2) {latent accumulation};
	\node[align=center] at (6.3,3.85) {recognized\\instability};
	
	\draw[arr] (5,2.65) -- (6.4,3.15) node[midway,above] {$\Delta$};
	\end{tikzpicture}
	\caption{Structural threshold crossing and delayed recognition. The functional $\Omega(\tau)$ crosses the threshold at $\tau^\ast$, while recognition occurs at $\tau^\ast+\Delta$.}
	\label{fig:threshold-delay}
\end{figure}

Figure \ref{fig:threshold-delay} illustrates the difference between structural crossing and recognized crossing. The scalar functional $\Omega(\tau)$ may exceed $\Theta$ before the system changes its interpretive or institutional response. This separation is consistent with the broader literature on tipping points and early-warning signals, where a system may lose resilience before a transition becomes visible in macroscopic behavior \citep{dakos2024tipping,bury2021deep,rietkerk2025ambiguity}.

\subsection{Delayed threshold recognition}

The delayed nature of the model can now be incorporated into the threshold functional. Since the state variables interact through delays, the perceived or institutionally effective instability need not depend on the current state $x(\tau)$ alone. Instead, one may define a recognized instability functional
\begin{equation}
\label{eq:omega-recognized}
\Omega_{\Delta}(\tau)
=
\Omega(\tau-\Delta),
\end{equation}
where $\Delta\geq 0$ is a recognition delay.

More generally, each component may be recognized with its own delay:
\begin{equation}
\label{eq:omega-component-delays}
\Omega_{\delta}(\tau)
=
\alpha C(\tau-\delta_C)
+
\beta H(\tau-\delta_H)
+
\gamma P(\tau-\delta_P)
-
\eta R(\tau-\delta_R)
-
\mu L(\tau-\delta_L).
\end{equation}
Here $\delta_C,\delta_H,\delta_P,\delta_R,\delta_L\geq 0$ represent component-specific recognition delays.

This distinction is important. The system may be structurally unstable according to $\Omega(\tau)$, while still being interpreted as stable according to $\Omega_{\delta}(\tau)$:
\[
\Omega(\tau)\geq \Theta
\quad
\text{but}
\quad
\Omega_{\delta}(\tau)<\Theta.
\]
In this case, the system has already crossed the structural threshold, but the delayed interpretive mechanism has not yet registered the crossing.

\begin{definition}[Latent instability]
	\label{def:latent-instability}
	The system is in a state of latent instability at time $\tau$ if
	$
	\Omega(\tau)\geq \Theta
	$
	while
$
\Omega_{\delta}(\tau)<\Theta.
$
\end{definition}

Latent instability is the formal counterpart of delayed recognition. It is the region in which the structural state of the system and the recognized state of the system disagree.


The threshold functional does not replace the delayed generalized Lotka--Volterra dynamics. Rather, it is defined on top of them. The state variables evolve according to  \eq{eq:general-delayed-lv}.
The functional $\Omega$ then maps the evolving state into a scalar instability index:
\[
x(\tau)
\longmapsto
\Omega(\tau).
\]
Thus, the complete structure of the model is:
\[
\text{delayed state dynamics} 
\Rightarrow 
\text{instability functional} 
\Rightarrow
\text{threshold crossing} 
\Rightarrow
\text{recognized transition}.
\]
The advantage of this construction is that it separates two questions. The delayed Lotka--Volterra system answers the dynamical question:
 How do the structural variables evolve? 
 The threshold functional answers the diagnostic question:
 When does their combined configuration become unstable?



\section{Equilibria and Positivity of Solutions}
\label{sec:equilibria-positivity}

We now study the elementary dynamical properties of the delayed generalized
Lotka--Volterra system introduced above. The aim of this section is not yet to
provide a complete stability theory, but to establish the basic phase-space
structure of the model: existence of equilibria, invariance of the non-negative
orthant, positivity of solutions, and the algebraic characterization of boundary
and interior equilibria.

We now study positivity, equilibria, and boundedness of system~\eqref{eq:five-variable-delayed-system}. This formulation is consistent with recent work on
delayed Lotka--Volterra and generalized Lotka--Volterra systems, where delays
are known to affect feasibility, permanence, boundedness, and local stability
\cite{saeedian2022delay,pigani2022delay,xu2022delayedlv,JiangHalikMuhammadhaji2023NSpeciesDelays}.

Let
\[
\delta_{\max}=\max_{1\leq i,j\leq n}\delta_{ij}.
\]
The natural phase space of \eqref{eq:five-variable-delayed-system} is
\[
\mathcal{C}:=
C\big([-\delta_{\max},0],\mathbb{R}^{n}\big),
\]
equipped with the supremum norm
\[
\|\varphi\|_{\infty}
=
\sup_{s\in[-\delta_{\max},0]}\|\varphi(s)\|.
\]
For an initial history \(\varphi\in\mathcal{C}\), the solution satisfies
\[
x_i(s)=\varphi_i(s),
\qquad s\in[-\delta_{\max},0],
\qquad i=1,\ldots,n.
\]
We denote by
\[
\mathcal{C}_{+}:=
C\big([-\delta_{\max},0],\mathbb{R}^{n}_{+}\big)
\]
the cone of non-negative histories, and by
\[
\mathcal{C}_{++}:=
C\big([-\delta_{\max},0],\mathbb{R}^{n}_{++}\big)
\]
the cone of strictly positive histories.

\subsection{Local existence and uniqueness}

We first record the standard local well-posedness statement. Although this result
is classical in the theory of functional differential equations
\cite{hale1993introduction,Kuang1993}, it is useful to state it explicitly
because all subsequent arguments require the solution to be defined on an
interval of chronological time.

\begin{theorem}[Local existence and uniqueness]
	\label{thm:local-existence}
	For every initial history \(\varphi\in\mathcal{C}\), there exists
	\(T_{\max}>0\) and a unique solution
	\[
	x:[-\delta_{\max},T_{\max})\to\mathbb{R}^{n}
	\]
	of \eqref{eq:general-delayed-lv} satisfying    
	\[
	x(s)=\varphi(s),
	\qquad s\in[-\delta_{\max},0].
	\]
	Moreover, the solution depends continuously on the initial history.
\end{theorem}

\begin{proof}
	Define the functional vector field \(F:\mathcal{C}\to\mathbb{R}^{n}\) by
	\[
	F_i(\phi)
	=
	\phi_i(0)
	\left(
	r_i+\sum_{j=1}^{n}a_{ij}\phi_j(-\delta_{ij})
	\right),
	\qquad i=1,\ldots,n.
	\]
	The map \(F\) is locally Lipschitz on \(\mathcal{C}\), since it is polynomial in
	finitely many evaluations of \(\phi\). The standard existence and uniqueness
	theorem for retarded functional differential equations therefore gives a unique
	local solution. Continuous dependence follows from the same theorem.
\end{proof}

\subsection{Invariance of the non-negative orthant}

The first structural property of \eqref{eq:general-delayed-lv} is that the
non-negative cone is positively invariant. This is essential in the present
model, since the variables represent non-negative quantities such as complexity,
pressure, resilience, and legitimacy.

\begin{theorem}[Positive invariance of \(\mathbb{R}^{n}_{+}\)]
	\label{thm:positive-invariance}
	Assume that the initial history satisfies
	\[
	\varphi\in\mathcal{C}_{+}.
	\]
	Then the corresponding solution of \eqref{eq:general-delayed-lv} satisfies
	\[
	x(\tau)\in\mathbb{R}^{n}_{+}
	\]
	for every \(\tau\in[0,T_{\max})\). Hence \(\mathbb{R}^{n}_{+}\) is positively
	invariant.
\end{theorem}

\begin{proof}
	For each component \(i\), write
	\[
	\dot{x}_i(\tau)=x_i(\tau)g_i(\tau),
	\]
	where
	\[
	g_i(\tau)
	=
	r_i+\sum_{j=1}^{n}a_{ij}x_j(\tau-\delta_{ij}).
	\]
	Suppose that \(x_i\) reaches zero for the first time at \(\tau_0>0\). Then
	\(x_i(\tau)>0\) for \(\tau<\tau_0\) sufficiently close to \(\tau_0\), and
	\(x_i(\tau_0)=0\). At \(\tau_0\), the differential equation gives
	\[
	\dot{x}_i(\tau_0)=x_i(\tau_0)g_i(\tau_0)=0.
	\]
	Thus the vector field is tangent to the boundary hyperplane \(x_i=0\), and the
	solution cannot cross from \(x_i\geq 0\) into \(x_i<0\). Since this argument holds
	for every component, the non-negative orthant is positively invariant.
\end{proof}

\begin{theorem}[Strict positivity]
	\label{thm:strict-positivity}
	Assume that the initial history satisfies
	\[
	\varphi\in\mathcal{C}_{++}.
	\]
	Then
	\[
	x_i(\tau)>0,
	\qquad i=1,\ldots,n,
	\]
	for every \(\tau\in[0,T_{\max})\).
\end{theorem}

\begin{proof}
	For each \(i\), the equation can be written as
	\[
	\frac{\dot{x}_i(\tau)}{x_i(\tau)}=g_i(\tau)
	\]
	whenever \(x_i(\tau)>0\). Integrating on \([0,\tau]\), we obtain
	\[
	x_i(\tau)
	=
	x_i(0)
	\exp\left(
	\int_{0}^{\tau}
	g_i(s)\,ds
	\right).
	\]
	Since \(x_i(0)=\varphi_i(0)>0\), and the exponential term is strictly positive,
	we have \(x_i(\tau)>0\) for all \(\tau\in[0,T_{\max})\).
\end{proof}

	The positivity result does not depend on the signs of the interaction
	coefficients \(a_{ij}\). It follows from the multiplicative Lotka--Volterra
	structure \(x_i(\tau)g_i(\tau)\). This is why delayed Lotka--Volterra models are
	well suited to variables that must remain non-negative. Recent studies of
	delayed ecological and cooperative systems also rely on positivity and
	permanence properties as preliminary steps before stability analysis
	\cite{xu2022delayedlv,JiangHalikMuhammadhaji2023NSpeciesDelays}.

\subsection{Equilibria}

An equilibrium of \eqref{eq:general-delayed-lv} is a constant solution
\[
x(\tau)\equiv x^\ast.
\]
Since \(x_j(\tau-\delta_{ij})=x_j^\ast\) for constant solutions, the delays do
not affect the algebraic equations defining equilibria. They affect stability,
but not the location of equilibria.

Substituting \(x(\tau)\equiv x^\ast\) into \eqref{eq:general-delayed-lv}, we
obtain
\begin{equation}
\label{eq:equilibrium-component}
0
=
x_i^\ast
\left(
r_i+\sum_{j=1}^{n}a_{ij}x_j^\ast
\right),
\qquad i=1,\ldots,n.
\end{equation}

Therefore, for each component \(i\), either
\[
x_i^\ast=0,
\]
or
\[
r_i+\sum_{j=1}^{n}a_{ij}x_j^\ast=0.
\]

\begin{definition}[Admissible equilibrium]
	An equilibrium \(x^\ast\in\mathbb{R}^{n}_{+}\) is called admissible if all of its
	components are non-negative. It is called an interior, or positive, equilibrium
	if
	\[
	x^\ast\in\mathbb{R}^{n}_{++}.
	\]
	It is called a boundary equilibrium if at least one component of \(x^\ast\) is
	zero.
\end{definition}

\begin{proposition}[Interior equilibrium]
	\label{prop:interior-equilibrium}
	Assume that \(x^\ast\in\mathbb{R}^{n}_{++}\). Then \(x^\ast\) is an equilibrium
	of \eqref{eq:general-delayed-lv} if and only if
	\begin{equation}
	\label{eq:interior-equilibrium}
	r+Ax^\ast=0,
	\end{equation}
	where
	\[
	r=(r_1,\ldots,r_n)^{T},
	\qquad
	A=(a_{ij})_{i,j=1}^{n}.
	\]
	If \(A\) is invertible, then the unique candidate interior equilibrium is
	\begin{equation}
	\label{eq:positive-eq-formula}
	x^\ast=-A^{-1}r.
	\end{equation}
	This equilibrium is admissible and interior if and only if
	\[
	-A^{-1}r\in\mathbb{R}^{n}_{++}.
	\]
\end{proposition}

\begin{proof}
	If \(x^\ast\in\mathbb{R}^{n}_{++}\), then \(x_i^\ast>0\) for every \(i\). Hence
	\eqref{eq:equilibrium-component} implies
	\[
	r_i+\sum_{j=1}^{n}a_{ij}x_j^\ast=0,
	\qquad i=1,\ldots,n.
	\]
	This is exactly \eqref{eq:interior-equilibrium}. If \(A\) is invertible, the
	linear system has the unique solution
	\[
	x^\ast=-A^{-1}r.
	\]
	The condition \(x^\ast\in\mathbb{R}^{n}_{++}\) is precisely the feasibility
	condition for an interior equilibrium.
\end{proof}

	The fact that the interior equilibrium is independent of the delays is important.
	The delays \(\delta_{ij}\) do not move \(x^\ast\), but they may change its
	stability. This distinction is one of the main reasons delayed generalized
	Lotka--Volterra systems are mathematically rich: feasibility and stability may
	separate once delays are introduced \cite{saeedian2022delay,pigani2022delay}.

\subsection{Equilibria of the five-variable civilizational model}

For the five-variable model
\[
x(\tau)=
\big(C(\tau),H(\tau),P(\tau),R(\tau),L(\tau)\big)^{T},
\]
an interior equilibrium
\[
x^\ast=(C^\ast,H^\ast,P^\ast,R^\ast,L^\ast)^{T}
\]
satisfies the algebraic system
\begin{equation}
\label{eq:five-variable-equilibrium}
\begin{cases}
r_C+a_{CH}H^\ast-a_{CR}R^\ast+a_{CP}P^\ast=0,\\
r_H+a_{HC}C^\ast-a_{HL}L^\ast-a_{HR}R^\ast=0,\\
r_P+a_{PC}C^\ast+a_{PH}H^\ast-a_{PR}R^\ast=0,\\
r_R-a_{RC}C^\ast-a_{RH}H^\ast-a_{RP}P^\ast+a_{RL}L^\ast=0,\\
r_L-a_{LC}C^\ast-a_{LH}H^\ast-a_{LP}P^\ast+a_{LR}R^\ast=0.
\end{cases}
\end{equation}

Equivalently,
\[
r+Ax^\ast=0,
\]
where
\[
r=
\begin{pmatrix}
r_C\\ r_H\\ r_P\\ r_R\\ r_L
\end{pmatrix},
\]
and
\[
A=
\begin{pmatrix}
0 & a_{CH} & a_{CP} & -a_{CR} & 0\\
a_{HC} & 0 & 0 & -a_{HR} & -a_{HL}\\
a_{PC} & a_{PH} & 0 & -a_{PR} & 0\\
-a_{RC} & -a_{RH} & -a_{RP} & 0 & a_{RL}\\
-a_{LC} & -a_{LH} & -a_{LP} & a_{LR} & 0
\end{pmatrix}.
\]
If \(A\) is invertible, the candidate interior equilibrium is
\[
x^\ast=-A^{-1}r.
\]
The feasibility condition is
\[
C^\ast>0,\qquad H^\ast>0,\qquad P^\ast>0,\qquad R^\ast>0,\qquad L^\ast>0.
\]

	In this model, feasibility has a direct interpretation. A positive equilibrium
	represents a regime in which complexity, entropy, pressure, resilience, and
	legitimacy coexist at non-zero levels. Such a state is not necessarily stable.
	Its stability depends on the eigenstructure of the delayed linearization, which
	will be studied in the next section.

At a positive equilibrium \(x^\ast\), the corresponding instability level is
\[
\Omega^\ast=\Omega(x^\ast).
\]
If \(\Omega^\ast<\Theta\), the equilibrium is structurally sub-threshold; if \(\Omega^\ast\geq\Theta\), it lies on or beyond the structural instability boundary. This classification is distinct from spectral stability: \(\Omega^\ast\) measures structural stress, while the characteristic roots of the delayed linearization determine the local dynamical response.

\subsection{A sufficient boundedness condition}

Positivity alone does not guarantee boundedness. Since generalized
Lotka--Volterra systems may display persistent oscillations or more complicated
dynamics, it is useful to impose a dissipativity condition. Similar dissipative
and permanence assumptions are common in the modern literature on delayed
Lotka--Volterra systems \cite{JiangHalikMuhammadhaji2023NSpeciesDelays}.

\begin{assumption}[Diagonal dissipativity]
	\label{ass:dissipativity}
	There exist constants \(b_i>0\) and \(m_i>0\) such that, for all admissible
	states,
	\[
	r_i+\sum_{j=1}^{n}a_{ij}x_j(\tau-\delta_{ij})
	\leq
	b_i-m_i x_i(\tau-\delta_{ii}),
	\qquad i=1,\ldots,n.
	\]
\end{assumption}

\begin{theorem}[Eventual boundedness under delayed self-limitation]
	\label{thm:eventual-boundedness}
	Assume that \(\varphi\in\mathcal{C}_{++}\) and that Assumption
	\ref{ass:dissipativity} holds. Then each component \(x_i(\tau)\) is eventually
	bounded above. More precisely, for every \(\varepsilon>0\), there exists
	\(T_\varepsilon>0\) such that
	\[
	x_i(\tau)\leq \frac{b_i}{m_i}+\varepsilon,
	\qquad
	\tau\geq T_\varepsilon,
	\qquad i=1,\ldots,n.
	\]
\end{theorem}

\begin{proof}
	Under Assumption \ref{ass:dissipativity},
	\[
	\dot{x}_i(\tau)
	\leq
	x_i(\tau)
	\left(
	b_i-m_i x_i(\tau-\delta_{ii})
	\right).
	\]
	The right-hand side is dominated by a delayed logistic-type inequality. By the
	standard comparison principle for scalar delay differential inequalities, the
	limsup of \(x_i(\tau)\) satisfies
	\[
	\limsup_{\tau\to\infty}x_i(\tau)\leq \frac{b_i}{m_i}.
	\]
	Therefore, for every \(\varepsilon>0\), there exists \(T_\varepsilon>0\) such
	that
	\[
	x_i(\tau)\leq \frac{b_i}{m_i}+\varepsilon
	\]
	for all \(\tau\geq T_\varepsilon\).
\end{proof}


Summarizing, the delayed generalized Lotka--Volterra model preserves the non-negative cone
and therefore remains consistent with the interpretation of the variables as
non-negative structural quantities. Equilibria are obtained from the same
algebraic equations as in the non-delayed model:
\[
r+Ax^\ast=0
\]
for interior equilibria. Hence delays do not change the location of equilibria.
They change the dynamical response around those equilibria.
The next section studies the delayed linearization around an admissible
equilibrium. This is where the delays \(\delta_{ij}\) become decisive: they enter
the characteristic equation and may destabilize an otherwise stable equilibrium.

 \section{Linear Stability and Delay--Induced Instability}

 \label{sec:linearization-local-stability}
 
 We now study the local behavior of the delayed generalized Lotka--Volterra
 system near a positive equilibrium. The goal of this section is twofold. First,
 we derive the linearized delayed system and the corresponding characteristic
 equation. Second, we show how delays may preserve local stability for small
 values but may destabilize the equilibrium when they become sufficiently large.
 This is consistent with the general theory of delay differential equations,
 where stability is determined by the roots of a transcendental characteristic
 equation rather than by the eigenvalues of a finite-dimensional matrix alone
 \cite{hale1993introduction,fukuda2022stability}.
 
 Consider the delayed generalized Lotka--Volterra system
 \eq{eq:general-delayed-lv}
 where $x_i(\tau)\geq 0$, $r_i\in\mathbb{R}$, $a_{ij}\in\mathbb{R}$, and
 $\delta_{ij}\geq 0$. We denote
 \[
 \delta_{\max}:=\max_{1\leq i,j\leq n}\delta_{ij}.
 \]
 The initial condition is given by a continuous history function
 \[
 x_i(\tau)=\varphi_i(\tau),
 \qquad
 \tau\in[-\delta_{\max},0],
 \qquad
 i=1,\ldots,n.
 \]
 
 Let $x^\ast=(x_1^\ast,\ldots,x_n^\ast)$ be a positive equilibrium, namely
 \[
 x_i^\ast>0,
 \qquad i=1,\ldots,n.
 \]
 Since $x_i^\ast>0$, the equilibrium condition is
 \begin{equation}
 \label{eq:positive-equilibrium-condition}
 r_i+\sum_{j=1}^{n}a_{ij}x_j^\ast=0,
 \qquad i=1,\ldots,n.
 \end{equation}
 Equivalently, in vector form,
 \begin{equation}
 \label{eq:equilibrium-vector-form}
 r+Ax^\ast=0.
 \end{equation}
 If $A$ is invertible, then
 \begin{equation}
 \label{eq:equilibrium-explicit}
 x^\ast=-A^{-1}r.
 \end{equation}
 The equilibrium is biologically, socially, or structurally meaningful only when
 \[
 -A^{-1}r\in\mathbb{R}^n_+.
 \]
 
 
 Define the perturbation variables
 \[
 u_i(\tau)=x_i(\tau)-x_i^\ast,
 \qquad i=1,\ldots,n.
 \]
 Then
 \[
 x_i(\tau)=x_i^\ast+u_i(\tau).
 \]
 Substituting this expression into \eqref{eq:general-delayed-lv}, we obtain
 \[
 \dot{u}_i(\tau)
 =
 \left(x_i^\ast+u_i(\tau)\right)
 \left[
 r_i+
 \sum_{j=1}^{n}
 a_{ij}
 \left(
 x_j^\ast+u_j(\tau-\delta_{ij})
 \right)
 \right].
 \]
 Using the equilibrium condition \eqref{eq:positive-equilibrium-condition}, this becomes
 \[
 \dot{u}_i(\tau)
 =
 \left(x_i^\ast+u_i(\tau)\right)
 \sum_{j=1}^{n}
 a_{ij}u_j(\tau-\delta_{ij}).
 \]
 Therefore, the linear part is
 \begin{equation}
 \label{eq:linearized-system}
 \dot{u}_i(\tau)
 =
 x_i^\ast
 \sum_{j=1}^{n}
 a_{ij}u_j(\tau-\delta_{ij}),
 \qquad i=1,\ldots,n.
 \end{equation}
 The nonlinear remainder is
 \[
 N_i(u_\tau)
 =
 u_i(\tau)
 \sum_{j=1}^{n}
 a_{ij}u_j(\tau-\delta_{ij}),
 \]
 which is of quadratic order near the origin.
 
 Let
 \[
 D^\ast:=\operatorname{diag}(x_1^\ast,\ldots,x_n^\ast).
 \]
 When all delays are zero, the linearized system reduces to the ordinary
 differential equation
 \[
 \dot{u}(\tau)=D^\ast A u(\tau).
 \]
 We define
 \begin{equation}
 \label{eq:B-matrix}
 B:=D^\ast A.
 \end{equation}
 Thus $B$ is the interaction matrix of the linearized non-delayed system.
 
 \begin{proposition}[Linearization]
 	\label{prop:linearization}
 	Let $x^\ast\in\mathbb{R}^n_+$ be a positive equilibrium of
 	\eqref{eq:general-delayed-lv}. Then the linearization of the delayed
 	generalized Lotka--Volterra system around $x^\ast$ is
 	\[
 	\dot{u}_i(\tau)
 	=
 	x_i^\ast
 	\sum_{j=1}^{n}
 	a_{ij}u_j(\tau-\delta_{ij}),
 	\qquad i=1,\ldots,n.
 	\]
 	In vector notation, this can be written as
 	\[
 	\dot{u}(\tau)=\mathcal{L}u_\tau,
 	\]
 	where $\mathcal{L}:C([-\delta_{\max},0],\mathbb{R}^n)\to\mathbb{R}^n$ is the
 	bounded linear operator defined componentwise by
 	\[
 	(\mathcal{L}\phi)_i
 	=
 	x_i^\ast
 	\sum_{j=1}^{n}
 	a_{ij}\phi_j(-\delta_{ij}).
 	\]
 \end{proposition}
 
 \begin{proof}
 	The result follows by expanding the right-hand side of
 	\eqref{eq:general-delayed-lv} around $x^\ast$ and using the equilibrium
 	condition \eqref{eq:positive-equilibrium-condition}. The first-order terms are
 	precisely those in \eqref{eq:linearized-system}, while the remaining terms are
 	quadratic in the perturbation variables.
 \end{proof}
 
 \subsection{Characteristic equation}
 
 We look for exponential solutions of the linearized system in the form
 \[
 u(\tau)=v e^{\lambda \tau},
 \qquad v\in\mathbb{C}^n\setminus\{0\}.
 \]
 Substitution into \eqref{eq:linearized-system} gives
 \[
 \lambda v_i
 =
 x_i^\ast
 \sum_{j=1}^{n}
 a_{ij}v_j e^{-\lambda\delta_{ij}},
 \qquad i=1,\ldots,n.
 \]
 Equivalently,
 \[
 \sum_{j=1}^{n}
 \left[
 \lambda\delta_{ij}^{K}
 -
 x_i^\ast a_{ij}e^{-\lambda\delta_{ij}}
 \right]v_j=0,
 \]
 where $\delta_{ij}^{K}$ denotes the Kronecker delta. Hence the characteristic
 equation is
 \begin{equation}
 \label{eq:characteristic-equation-general}
 \Delta(\lambda)
 :=
 \det
 \left[
 \lambda I
 -
 M(\lambda)
 \right]
 =0,
 \end{equation}
 where
 \begin{equation}
 \label{eq:M-lambda}
 M_{ij}(\lambda)
 =
 x_i^\ast a_{ij}e^{-\lambda\delta_{ij}}.
 \end{equation}
 
 \begin{definition}[Spectral bound]
 	The spectral bound of the linearized delayed system is
 	\[
 	s(\mathcal{L})
 	:=
 	\sup
 	\left\{
 	\operatorname{Re}\lambda:
 	\Delta(\lambda)=0
 	\right\}.
 	\]
 	The equilibrium $x^\ast$ is said to be linearly asymptotically stable if
 	\[
 	s(\mathcal{L})<0.
 	\]
 \end{definition}
 
 The characteristic equation \eqref{eq:characteristic-equation-general} is
 transcendental whenever at least one delay is positive. Consequently, it has
 infinitely many complex roots in general. This is the essential difference
 between the delayed and non-delayed systems. In the non-delayed case, stability
 is determined by finitely many eigenvalues of $B=D^\ast A$. In the delayed
 case, stability is determined by the location of infinitely many characteristic
 roots \cite{hale1993introduction,fukuda2022stability}.
 
 \begin{theorem}[Linearized stability principle]
 	\label{thm:linearized-stability-principle}
 	Let $x^\ast$ be a positive equilibrium of
 	\eqref{eq:five-variable-delayed-system}. If all roots of the characteristic equation
 	\[
 	\Delta(\lambda)=0
 	\]
 	satisfy
 	\[
 	\operatorname{Re}\lambda<0,
 	\]
 	then $x^\ast$ is locally asymptotically stable. If at least one characteristic
 	root satisfies
 	\[
 	\operatorname{Re}\lambda>0,
 	\]
 	then $x^\ast$ is unstable.
 \end{theorem}
 
 \begin{proof}
 	This is the standard linearized stability principle for retarded functional
 	differential equations. The nonlinear terms in the perturbation equation are
 	quadratic near the equilibrium, while the linear part generates the local
 	asymptotic behavior. If the spectral bound of the linearized equation is
 	negative, the zero solution of the perturbation equation decays exponentially.
 	If the spectral bound is positive, an unstable mode grows exponentially. See,
 	for example, the general theory of retarded functional differential equations
 	\cite{hale1993introduction} and recent multidimensional stability criteria for
 	linear delay systems \cite{fukuda2022stability}.
 \end{proof}
 
 \subsection{Small--Delay Stability}
 
 The first useful consequence is that if the non-delayed equilibrium is strongly
 stable, then sufficiently small delays do not immediately destroy local
 stability. This is important for the interpretation of the model: small
 recognition delays may be absorbed by the system, whereas large delays may
 produce destabilization.
 
 \begin{theorem}[Small-delay stability]
 	\label{thm:small-delay-stability}
 	Assume that the matrix
 	\[
 	B=D^\ast A
 	\]
 	is Hurwitz, namely
 	\[
 	\operatorname{Re}\lambda_k(B)<0
 	\qquad
 	\text{for all eigenvalues } \lambda_k(B).
 	\]
 	Then there exists $\varepsilon>0$ such that, if
 	\[
 	0\leq \delta_{ij}<\varepsilon
 	\qquad
 	\text{for all } i,j,
 	\]
 	the positive equilibrium $x^\ast$ of the delayed system
 	\eqref{eq:five-variable-delayed-system} is locally asymptotically stable.
 \end{theorem}
 
 \begin{proof}
 	When $\delta_{ij}=0$ for all $i,j$, the characteristic equation becomes
 	\[
 	\det(\lambda I-B)=0.
 	\]
 	By assumption, all roots lie in the open left half-plane. Since characteristic
 	roots of retarded delay differential equations depend continuously on the
 	delays on bounded subsets of the complex plane, sufficiently small perturbations
 	of the delays cannot move any characteristic root across the imaginary axis.
 	Moreover, for retarded equations, no roots can enter the right half-plane from
 	infinity under sufficiently small delay perturbations. Therefore, there exists
 	$\varepsilon>0$ such that all roots of \eqref{eq:characteristic-equation-general}
 	remain in the open left half-plane whenever
 	$0\leq\delta_{ij}<\varepsilon$. The conclusion follows from
 	Theorem~\ref{thm:linearized-stability-principle}.
 \end{proof}

 	The previous theorem should not be read as saying that delay is harmless.
 	Rather, it says that local stability is robust only for sufficiently small
 	delays. Once delays become large enough, the exponential factors
 	$e^{-\lambda\delta_{ij}}$ may move characteristic roots toward or across the
 	imaginary axis. This is precisely the mathematical expression of delayed
 	recognition becoming dynamically dangerous.

 \subsection{Common-delay case and Lambert \(W\) representation}
 
 A more explicit analysis is possible when all interaction delays are equal:
 \[
 \delta_{ij}=\delta
 \qquad
 \text{for all } i,j.
 \]
 Then the linearized system becomes
 \begin{equation}
 \label{eq:common-delay-linearized}
 \dot{u}(\tau)=B u(\tau-\delta),
 \end{equation}
 where $B=D^\ast A$. The characteristic equation is \eq{eq:common-delay-characteristic}.
 If $\rho_1,\ldots,\rho_n$ are the eigenvalues of $B$, then
 \eqref{eq:common-delay-characteristic} is equivalent to
 \begin{equation}
 \label{eq:scalar-characteristic-rho}
 \lambda=\rho_k e^{-\lambda\delta},
 \qquad k=1,\ldots,n.
 \end{equation}
 Thus
 \[
 \lambda e^{\lambda\delta}=\rho_k.
 \]
 Multiplying by $\delta$, we obtain
 \[
 \lambda\delta e^{\lambda\delta}=\rho_k\delta.
 \]
 Therefore,
 \begin{equation}
 \label{eq:lambda-lambert}
 \lambda_{k,m}(\delta)
 =
 \frac{1}{\delta}
 W_m(\rho_k\delta),
 \qquad
 m\in\mathbb{Z},
 \end{equation}
 where $W_m$ is the $m$-th branch of the Lambert $W$ function. 

 The Lambert \(W\) function is defined as the inverse function of
 \[
 f(W)=W e^{W}.
 \]
 Thus, for a complex number \(z\), the Lambert function \(W(z)\) is defined implicitly by
 \[
 W(z)e^{W(z)}=z.
 \]
 Equivalently,
 \[
 W(z)=w
 \quad \Longleftrightarrow \quad
 w e^w=z.
 \]
 
 Since the map \(w\mapsto we^w\) is not one-to-one on the complex plane, the Lambert function is multivalued. Its branches are denoted by
 \[
 W_k(z), \qquad k\in\mathbb{Z}.
 \]
 The principal branch is denoted by
 \[
 W_0(z).
 \]
 
 In the context of delay differential equations, the Lambert \(W\) function appears naturally when solving characteristic equations of the form
 \[
 \lambda=\rho e^{-\lambda\delta}.
 \]
 Multiplying by \(e^{\lambda\delta}\), one obtains
 \[
 \lambda e^{\lambda\delta}=\rho.
 \]
 Multiplying by \(\delta\), this becomes
 \[
 \lambda\delta e^{\lambda\delta}=\rho\delta.
 \]
 By the definition of the Lambert \(W\) function,
 \[
 \lambda\delta = W_k(\rho\delta),
 \]
 and therefore
 \[
 \lambda=\frac{1}{\delta}W_k(\rho\delta),
 \qquad k\in\mathbb{Z}.
 \]
 Thus, each branch \(W_k\) gives a possible characteristic root of the delayed system.

 Lambert
 \(W\)-based representations are often useful for the spectral analysis of
 linear delay equations, especially when reducing a delayed characteristic
 equation to explicit root branches \cite{ruzgas2024lambert}.
 
 \begin{theorem}[Common-delay spectral criterion]
 	\label{thm:common-delay-spectral}
 	Assume that all delays are equal to $\delta>0$. Let
 	$\rho_1,\ldots,\rho_n$ be the eigenvalues of $B=D^\ast A$. Then the positive
 	equilibrium $x^\ast$ is locally asymptotically stable if and only if
 	\[
 	\operatorname{Re}
 	\left[
 	\frac{1}{\delta}W_m(\rho_k\delta)
 	\right]<0
 	\]
 	for every eigenvalue $\rho_k$ of $B$ and for every branch
 	$m\in\mathbb{Z}$ of the Lambert $W$ function.
 \end{theorem}
 
 \begin{proof}
 	For a common delay, the characteristic equation separates into the scalar
 	equations
 	\[
 	\lambda=\rho_k e^{-\lambda\delta},
 	\qquad k=1,\ldots,n.
 	\]
 	Each scalar equation has the root representation
 	\[
 	\lambda_{k,m}(\delta)
 	=
 	\frac{1}{\delta}W_m(\rho_k\delta),
 	\qquad m\in\mathbb{Z}.
 	\]
 	The equilibrium is locally asymptotically stable precisely when all
 	characteristic roots have negative real part. This gives the stated condition.
 \end{proof}

 \subsection{Delay-induced loss of stability}
 
 The common-delay case shows clearly how a stable non-delayed equilibrium may
 lose stability when delay increases. To make this mechanism explicit, consider
 the case in which $B$ has a real negative eigenvalue
 \[
 \rho=-a,
 \qquad a>0.
 \]
 The corresponding scalar characteristic equation is
 \[
 \lambda=-a e^{-\lambda\delta},
 \]
 or equivalently
 \begin{equation}
 \label{eq:scalar-negative-characteristic}
 \lambda+a e^{-\lambda\delta}=0.
 \end{equation}
 
 \begin{theorem}[Critical delay for a real negative mode]
 	\label{thm:critical-delay-real-negative}
 	Let $\rho=-a$, with $a>0$, be a real negative eigenvalue of $B$. The scalar
 	mode associated with $\rho$ is asymptotically stable for
 	\[
 	0\leq \delta<\frac{\pi}{2a}.
 	\]
 	At the critical value
 	\[
 	\delta^\ast=\frac{\pi}{2a},
 	\]
 	the characteristic equation has a pair of purely imaginary roots
 	\[
 	\lambda=\pm ia.
 	\]
 	Moreover, this crossing is transversal.
 \end{theorem}
 
 \begin{proof}
 	Substituting $\lambda=i\omega$ into
 	\eqref{eq:scalar-negative-characteristic} gives
 	\[
 	i\omega+a e^{-i\omega\delta}=0.
 	\]
 	Separating real and imaginary parts yields
 	\[
 	a\cos(\omega\delta)=0,
 	\]
 	and
 	\[
 	\omega-a\sin(\omega\delta)=0.
 	\]
 	Thus
 	\[
 	\omega\delta=\frac{\pi}{2}
 	\quad
 	\text{and}
 	\quad
 	\omega=a.
 	\]
 	Hence
 	\[
 	\delta^\ast=\frac{\pi}{2a}.
 	\]
 	To verify transversality, differentiate
 	\[
 	\lambda=\rho e^{-\lambda\delta}
 	\]
 	with respect to $\delta$. Since $\rho e^{-\lambda\delta}=\lambda$, we obtain
 	\[
 	\frac{d\lambda}{d\delta}
 	=
 	\lambda
 	\left(
 	-\delta\frac{d\lambda}{d\delta}
 	-\lambda
 	\right).
 	\]
 	Therefore,
 	\[
 	\frac{d\lambda}{d\delta}
 	=
 	-\frac{\lambda^2}{1+\lambda\delta}.
 	\]
 	At $\lambda=ia$ and $\delta=\pi/(2a)$,
 	\[
 	\operatorname{Re}
 	\left(
 	\frac{d\lambda}{d\delta}
 	\right)
 	=
 	\frac{a^2}{1+a^2(\delta^\ast)^2}
 	>0.
 	\]
 	The characteristic roots cross the imaginary axis from left to right as
 	$\delta$ increases. Hence the crossing is transversal.
 \end{proof}

 	Theorem~\ref{thm:critical-delay-real-negative} gives a precise mathematical
 	meaning to delay-induced instability. Even a stabilizing negative mode can
 	become oscillatory and unstable if its effect is delayed beyond a critical
 	threshold. In the civilizational interpretation, a stabilizing institutional
 	response may cease to stabilize the system if it arrives too late.


 The spectral mechanism of delay-induced instability is shown in
 Figure~\ref{fig:root-crossing}. For small delays, all characteristic roots lie
 in the left half-plane. At a critical delay, a conjugate pair crosses the
 imaginary axis, producing oscillatory instability.
 
 \begin{figure}[ht]
 	\centering
 	\begin{tikzpicture}[scale=1.1]
 	\draw[->] (-3,0) -- (3,0) node[right] {$\operatorname{Re}\lambda$};
 	\draw[->] (0,-2.2) -- (0,2.2) node[above] {$\operatorname{Im}\lambda$};
 	
 	\draw[thick,->] (-1.8,1.2) .. controls (-0.9,1.1) and (-0.25,0.85) .. (0,0.9);
 	\draw[thick,->] (-1.8,-1.2) .. controls (-0.9,-1.1) and (-0.25,-0.85) .. (0,-0.9);
 	
 	\filldraw (-1.8,1.2) circle (1.5pt);
 	\filldraw (-1.8,-1.2) circle (1.5pt);
 	\filldraw (0,0.9) circle (1.5pt);
 	\filldraw (0,-0.9) circle (1.5pt);
 	
 	\node at (-2.25,1.45) {$\delta<\delta^\ast$};
 	\node at (1.25,1.0) {$\lambda=ia$};
 	\node at (1.25,-1.0) {$\lambda=-ia$};
 	\node at (0.6,0.25) {$\delta=\delta^\ast$};
 	
 	\draw[dashed] (0,0.9) -- (0,-0.9);
 	\end{tikzpicture}
 	\caption{Schematic root crossing caused by increasing delay. A stable pair of
 		roots moves toward the imaginary axis and crosses it at the critical delay
 		$\delta^\ast$, generating oscillatory instability.}
 	\label{fig:root-crossing}
 \end{figure}
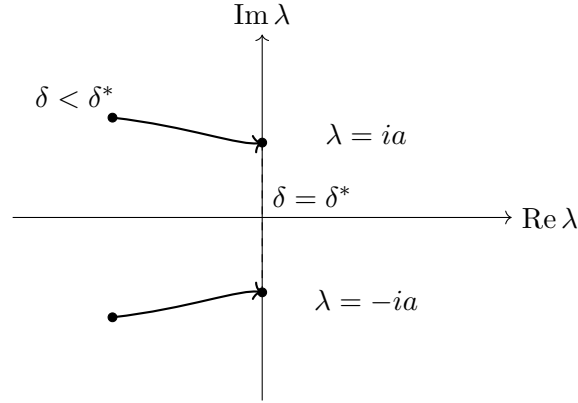
 
 The theorem shows that a stabilizing negative mode may become destabilizing solely because of delay. In the language of the present model, even if the instantaneous effect of resilience, legitimacy, or institutional correction is stabilizing, excessive delay may turn the correction mechanism into a source of oscillatory instability.

\subsection{Multidimensional sufficient condition}

The previous theorem can be extended to a common-delay multidimensional system when the spectrum of $B$ is real and negative.

\begin{theorem}[Sufficient stability condition for real negative modes]
	\label{thm:multidimensional-critical-delay}
	Assume that all delays are equal to $\delta>0$ and that $B=D(x^\ast)A$ is diagonalizable with real negative eigenvalues
	\[
	\rho_k=-a_k,
	\qquad a_k>0,
	\qquad k=1,\ldots,n.
	\]
	If
	\[
	0\leq \delta < \min_{1\leq k\leq n}\frac{\pi}{2a_k},
	\]
	then the equilibrium $x^\ast$ is locally asymptotically stable.
	Moreover, if for some $k$,
	\[
	\delta=\frac{\pi}{2a_k},
	\]
	then the corresponding spectral mode reaches a Hopf-type stability boundary.
\end{theorem}

\begin{proof}
	Since $B$ is diagonalizable, the common-delay linearized system can be decomposed into modal equations of the form
	\[
	\dot y_k(\tau)=\rho_k y_k(\tau-\delta)
	=
	-a_k y_k(\tau-\delta).
	\]
	By Theorem \ref{thm:critical-delay-real-negative}, the $k$-th mode is asymptotically stable whenever
	\[
	\delta<\frac{\pi}{2a_k}.
	\]
	Therefore, all modes are asymptotically stable if
	\[
	\delta<\min_{1\leq k\leq n}\frac{\pi}{2a_k}.
	\]
	At equality for one mode, the associated pair of characteristic roots lies on the imaginary axis, giving a Hopf-type stability boundary.
\end{proof}

	Theorem \ref{thm:multidimensional-critical-delay} is intentionally conservative. It assumes real negative spectral modes and a common delay. The general case with complex eigenvalues or multiple interaction delays requires direct analysis of the transcendental characteristic equation \eqref{eq:scalar-negative-characteristic}. Nevertheless, the theorem captures the basic mechanism: delay can destabilize a system that is stable in the absence of delay.

By Theorem~\ref{thm:critical-delay-real-negative}, the scalar stability boundary is transversal. Hence, under the usual nondegeneracy assumptions, the critical delay may be interpreted as a Hopf-type transition from monotone stabilization to oscillatory response.

 \subsection{Interpretation for the instability functional}
 
 The local stability of the equilibrium $x^\ast$ has a direct implication for
 the instability functional
 \[
 \Omega(\tau)
 =
 \alpha C(\tau)
 +
 \beta H(\tau)
 +
 \gamma P(\tau)
 -
 \eta R(\tau)
 -
 \mu L(\tau).
 \]
 If $x^\ast$ is locally asymptotically stable and
 \[
 \Omega(x^\ast)<\Theta,
 \]
 then trajectories starting sufficiently close to $x^\ast$ remain below the
 instability threshold for sufficiently small perturbations. In contrast, if
 delay destabilizes $x^\ast$, then the variables may oscillate or drift away
 from the stable configuration. In that case, even if
 \[
 \Omega(x^\ast)<\Theta,
 \]
 there may exist times $\tau$ such that
 \[
 \Omega(\tau)\geq \Theta.
 \]
 
 Thus local spectral stability provides a mathematical condition under which
 delayed recognition remains controlled. Loss of spectral stability provides a
 mechanism by which delayed recognition becomes a source of structural
 instability. This connects the local analysis of the delayed differential
 system with the threshold interpretation of civilizational instability.
 
 \begin{proposition}[Local threshold safety]
 	\label{prop:local-threshold-safety}
 	Assume that $x^\ast$ is locally asymptotically stable and that
 	\[
 	\Omega(x^\ast)<\Theta.
 	\]
 	Then there exists a neighborhood $U$ of $x^\ast$ such that, for every initial
 	history sufficiently close to $x^\ast$, the corresponding solution satisfies
 	\[
 	\Omega(x(\tau))<\Theta
 	\]
 	for all sufficiently large $\tau$.
 \end{proposition}
 
 \begin{proof}
 	Since $x^\ast$ is locally asymptotically stable, solutions starting sufficiently
 	close to $x^\ast$ converge to $x^\ast$. Since $\Omega$ is continuous and
 	$\Omega(x^\ast)<\Theta$, there exists a neighborhood $U$ of $x^\ast$ such that
 	\[
 	\Omega(x)<\Theta
 	\qquad
 	\text{for all } x\in U.
 	\]
 	For initial histories sufficiently close to $x^\ast$, the solution eventually
 	enters and remains in $U$. Therefore,
 	\[
 	\Omega(x(\tau))<\Theta
 	\]
 	for all sufficiently large $\tau$.
 \end{proof}

 	Proposition~\ref{prop:local-threshold-safety} does not exclude transient
 	threshold crossings. If the initial perturbation is large, or if the convergence
 	toward equilibrium is oscillatory, the instability functional may temporarily
 	approach or exceed the threshold before returning to the stable regime.
 	Therefore, in applications, local asymptotic stability should be complemented
 	by numerical simulations and basin-of-attraction estimates.

The delayed generalized Lotka--Volterra system may be locally stable near $x^\ast$, while the functional $\Omega$ remains below the threshold $\Theta$. However, delay-induced instability creates a mechanism through which oscillations in the variables may drive $\Omega$ above the threshold.

Let
\[
\Omega^\ast=\Omega(x^\ast).
\]
Assume initially that
\[
\Omega^\ast<\Theta.
\]
A small perturbation around $x^\ast$ gives
\[
\Omega(\tau)-\Omega^\ast
=
\nabla\Omega(x^\ast)\cdot u(\tau)
+
o(\|u(\tau)\|),
\]
where
\[
\nabla\Omega
=
(\alpha,\beta,\gamma,-\eta,-\mu).
\]
If the delayed linearized system is stable, then
\[
u(\tau)\to 0
\]
and hence
\[
\Omega(\tau)\to \Omega^\ast<\Theta.
\]
If delay destabilizes the equilibrium, however, oscillations or growth of $u(\tau)$ may lead to
\[
\Omega(\tau)\geq \Theta
\]
for some $\tau$. Thus, the delay does not only affect local stability. It may determine whether the instability threshold is eventually crossed.

\begin{proposition}[Delay-induced threshold crossing]
	\label{prop:delay-threshold-crossing}
	Let $x^\ast$ be a positive equilibrium satisfying
	\[
	\Omega(x^\ast)<\Theta.
	\]
	Assume that the non-delayed linearization is asymptotically stable, but that for some delay value $\delta>\delta_c$ the delayed linearization has a characteristic root with positive real part. If the unstable mode has nonzero projection on $\nabla\Omega(x^\ast)$, then arbitrarily small perturbations may generate trajectories for which
	\[
	\Omega(\tau)\geq\Theta
	\]
	at some later time, provided the perturbation persists within the validity range of the linear approximation long enough to amplify.
\end{proposition}

\begin{proof}
	Let $\lambda$ be a characteristic root of the delayed linearized system with
	\[
	\operatorname{Re}\lambda>0.
	\]
	Then there exists a solution component of the linearized system of the form
	\[
	u(\tau)=ve^{\lambda\tau}+\overline{v}e^{\overline{\lambda}\tau}
	\]
	or, in the real-root case,
	\[
	u(\tau)=ve^{\lambda\tau}.
	\]
	Since
	\[
	\nabla\Omega(x^\ast)\cdot v\neq 0,
	\]
	the first-order variation
	\[
	\Omega(\tau)-\Omega^\ast
	=
	\nabla\Omega(x^\ast)\cdot u(\tau)
	+
	o(\|u(\tau)\|)
	\]
	contains a component growing like $e^{\operatorname{Re}\lambda \tau}$. Therefore, for sufficiently long time within the linear regime, the deviation of $\Omega(\tau)$ from $\Omega^\ast$ can become large enough to reach $\Theta$. This proves the claim.
\end{proof}

 For small delays, characteristic roots remain in the left half-plane. At the critical delay, a conjugate pair crosses the imaginary axis. For larger delays, the equilibrium becomes unstable (see Fig. \ref{fig:delay-threshold-functional}).

Figure \ref{fig:delay-threshold-functional} shows how the delayed instability mechanism interacts with the threshold functional $\Omega$. The structural variables may oscillate or amplify after a delay-induced loss of stability, eventually pushing $\Omega$ above the critical threshold $\Theta$.

\begin{figure}[ht]
	\centering
	\begin{tikzpicture}[>=Latex,node distance=2.1cm]
	
	\node[draw,rounded corners,align=center,minimum width=2.6cm,minimum height=1cm] (stable) {Stable equilibrium\\$x^\ast$};
	\node[draw,rounded corners,align=center,right=of stable,minimum width=2.9cm,minimum height=1cm] (delay) {Increasing delay\\$\delta\uparrow$};
	\node[draw,rounded corners,align=center,right=of delay,minimum width=3.0cm,minimum height=1cm] (roots) {Root crossing\\$\operatorname{Re}\lambda>0$};
	\node[draw,rounded corners,align=center,below=of roots,minimum width=3.2cm,minimum height=1cm] (omega) {Amplification of\\$\Omega(\tau)-\Omega^\ast$};
	\node[draw,rounded corners,align=center,left=of omega,minimum width=3.2cm,minimum height=1cm] (threshold) {Threshold crossing\\$\Omega(\tau)\geq\Theta$};
	\node[draw,rounded corners,align=center,left=of threshold,minimum width=2.8cm,minimum height=1cm] (recognition) {Recognized\\instability};
	
	\draw[->,thick] (stable) -- (delay);
	\draw[->,thick] (delay) -- (roots);
	\draw[->,thick] (roots) -- (omega);
	\draw[->,thick] (omega) -- (threshold);
	\draw[->,thick] (threshold) -- (recognition);
	
	\draw[->,thick,dashed,bend left=35] (recognition.north) to node[above,align=center] {feedback\\reaction} (stable.south);
	
	\end{tikzpicture}
	\caption{Delay-induced instability and threshold crossing. Increasing delay may destabilize the equilibrium, amplify the instability functional, and generate recognized instability.}
	\label{fig:delay-threshold-functional}
\end{figure}
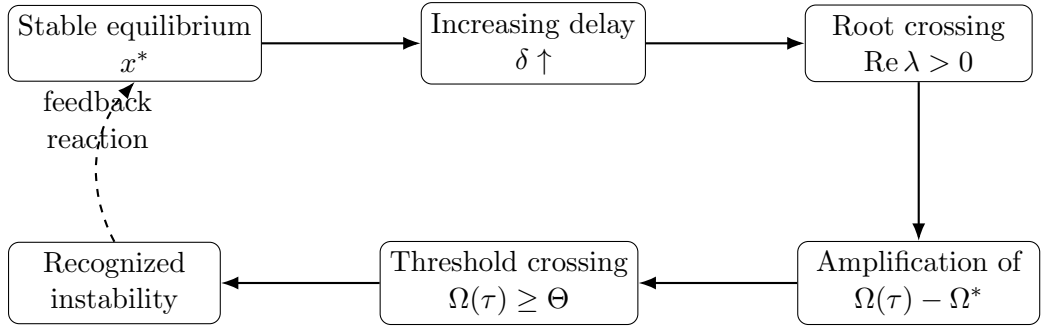

\subsection{Consequences for delayed recognition}

The mathematical results support three conclusions.

First, delay can destabilize an otherwise stable equilibrium. The non-delayed system may return to equilibrium, while the delayed system may generate oscillations or instability.

Second, the critical delay acts as a recognition boundary. Below the critical delay, perturbations are absorbed. Above it, the response is too late to stabilize the system.

Third, delay-induced instability provides a mechanism for threshold crossing. The instability functional $\Omega$ may remain below $\Theta$ near equilibrium, but oscillations generated by delay may push it above the critical threshold.

Thus, delayed recognition is not only a philosophical or interpretive concept. It has a precise dynamical counterpart: delays alter the characteristic spectrum of the linearized system and may transform local stability into instability.


 
 When a conjugate pair of characteristic roots crosses the imaginary axis as the delay parameter varies, the delayed system may undergo a Hopf bifurcation. In the present framework, this corresponds to a transition from monotone stabilization to delayed oscillatory correction. Larger delays or stronger nonlinear feedbacks may generate irregular regimes. We do not claim the existence of a chaotic attractor in the present paper; we only identify delayed feedback as a mechanism compatible with complex post-bifurcation dynamics.


\section{Numerical Scenarios}

This section presents four illustrative numerical scenarios for the delayed five-variable system introduced above. The aim is not to calibrate the model on historical data, but to show how different combinations of interaction coefficients and delays may generate qualitatively distinct regimes: stable adaptation, delayed oscillations, threshold crossing, and collapse-like transition.

We consider the normalized state vector
\[
x(\tau)=\big(C(\tau),H(\tau),P(\tau),R(\tau),L(\tau)\big),
\]
where \(C\) denotes systemic complexity, \(H\) informational entropy, \(P\) systemic pressure, \(R\) resilience, and \(L\) legitimacy. The numerical experiments are based on the delayed generalized Lotka--Volterra system
\[
\dot{x}_i(\tau)
=
x_i(\tau)
\left(
r_i
+
\sum_{j=1}^{5}a_{ij}x_j(\tau-\delta_{ij})
\right),
\qquad i=1,\ldots,5.
\]
For numerical boundedness, diagonal self-regulation terms are included:
\[
a_{ii}<0,
\]
so that each variable is prevented from growing indefinitely in the absence of compensating feedback. This is standard in generalized Lotka--Volterra simulations, where diagonal damping is often used to guarantee biologically or structurally meaningful bounded trajectories \cite{saeedian2022delay,liu2023feasibility}.

The instability functional is defined as
\[
\Omega(\tau)
=
\alpha C(\tau)
+
\beta H(\tau)
+
\gamma P(\tau)
-
\eta R(\tau)
-
\mu L(\tau),
\]
with
\[
\alpha=0.35,\qquad
\beta=0.30,\qquad
\gamma=0.35,\qquad
\eta=0.30,\qquad
\mu=0.25.
\]
The instability threshold is fixed at
\[
\Theta=0.25.
\]
Thus, the system is interpreted as structurally stable while
\[
\Omega(\tau)<\Theta,
\]
and as structurally unstable once
\[
\Omega(\tau)\geq\Theta.
\]

The initial history is assumed constant on the interval \([-\delta_{\max},0]\):
\[
C(\tau)=0.45,\quad
H(\tau)=0.35,\quad
P(\tau)=0.30,\quad
R(\tau)=0.70,\quad
L(\tau)=0.75.
\]
Therefore,
\[
\Omega(0)
=
0.35(0.45)+0.30(0.35)+0.35(0.30)-0.30(0.70)-0.25(0.75)
=
-0.03.
\]
The system initially lies below the instability threshold.

The simulations are dimensionless. The variable \(\tau\) should be interpreted as a normalized chronological time, while the delays \(\delta_{ij}\) represent institutional, cognitive, informational, or adaptive response lags. The qualitative relevance of the simulations lies in the change of dynamical regime as delays and interaction strengths vary, a phenomenon widely observed in delayed differential equations and delayed Lotka--Volterra systems \cite{rihan2021dde,saeedian2022delay,chenjiang2021turinghopf}.

\subsection{Scenario 1: Stable adaptation}

The first scenario describes a system in which destabilizing variables are present, but stabilizing feedback is sufficiently strong and sufficiently fast. Complexity, entropy, and pressure remain controlled by resilience and legitimacy.

The parameter set is shown in Table~\ref{tab:scenario1}. The notation \(b_i=-a_{ii}>0\) denotes self-regulation.

\begin{table}[h!]
	\centering
	\caption{Scenario 1: stable adaptation.}
	\label{tab:scenario1}
	\begin{tabular}{lccccc}
		\toprule
		Parameter & \(C\) & \(H\) & \(P\) & \(R\) & \(L\) \\
		\midrule
		Intrinsic rate \(r_i\) & \(0.03\) & \(0.02\) & \(0.025\) & \(0.04\) & \(0.035\) \\
		Self-regulation \(b_i\) & \(0.12\) & \(0.15\) & \(0.14\) & \(0.10\) & \(0.11\) \\
		Common delay \(\delta\) & \(1\) & \(1\) & \(1\) & \(1\) & \(1\) \\
		\bottomrule
	\end{tabular}
\end{table}

The main interaction coefficients are chosen as
\[
a_{CH}=0.06,\quad
a_{CP}=0.04,\quad
a_{HC}=0.05,\quad
a_{PC}=0.06,\quad
a_{PH}=0.06,
\]
for destabilizing interactions, and
\[
a_{CR}=0.05,\quad
a_{HL}=0.06,\quad
a_{HR}=0.04,\quad
a_{PR}=0.08,
\]
for stabilizing feedback. Resilience and legitimacy mutually reinforce each other through
\[
a_{RL}=0.08,
\qquad
a_{LR}=0.07.
\]

Numerically, the instability functional decreases from
\[
\Omega(0)=-0.03
\]
to approximately
\[
\Omega(100)\approx -0.40.
\]
A representative numerical profile is given in Table~\ref{tab:omega_stable}.

\begin{table}[h!]
	\centering
	\caption{Scenario 1: representative values of \(\Omega(\tau)\).}
	\label{tab:omega_stable}
	\begin{tabular}{cccccccccc}
		\toprule
		\(\tau\) & 0 & 10 & 20 & 30 & 40 & 50 & 60 & 80 & 100 \\
		\midrule
		\(\Omega(\tau)\) 
		& \(-0.030\) & \(-0.057\) & \(-0.085\) & \(-0.115\) & \(-0.147\) & \(-0.183\) & \(-0.222\) & \(-0.308\) & \(-0.399\) \\
		\bottomrule
	\end{tabular}
\end{table}

The interpretation is straightforward. The system adapts to pressure without crossing the instability threshold. In dynamical terms, the delayed feedback is not large enough to destabilize the equilibrium. In structural terms, resilience and legitimacy absorb the growth of complexity and entropy.

\subsection{Scenario 2: Delayed oscillations}

The second scenario illustrates a system in which stabilizing feedback is still present, but delayed. The delay does not immediately destroy stability, but it produces persistent oscillations around the stable regime. This is a standard effect in delay differential equations: feedback that would stabilize an ordinary differential equation may generate oscillations when it acts after a sufficiently large delay \cite{rihan2021dde}. In delayed Lotka--Volterra systems, increasing delay may move characteristic roots across the imaginary axis, producing Hopf-type transitions \cite{saeedian2022delay,chenjiang2021turinghopf}.

We use the same initial condition as before, but increase the common delay:
\[
\delta_{ij}=8,
\qquad i,j=1,\ldots,5.
\]
The destabilizing coefficients are moderately amplified:
\[
a_{CH}=0.08,\quad
a_{CP}=0.06,\quad
a_{HC}=0.07,\quad
a_{PC}=0.08,\quad
a_{PH}=0.08.
\]
The stabilizing feedback remains active, but it is not instantaneous:
\[
a_{CR}=0.05,\quad
a_{HL}=0.06,\quad
a_{HR}=0.04,\quad
a_{PR}=0.08.
\]

A representative oscillatory trajectory of the instability functional is given in Table~\ref{tab:omega_oscillatory}.

\begin{table}[h!]
	\centering
	\caption{Scenario 2: delayed oscillations around the threshold region.}
	\label{tab:omega_oscillatory}
	\begin{tabular}{cccccccccc}
		\toprule
		\(\tau\) & 0 & 10 & 20 & 30 & 40 & 50 & 60 & 80 & 100 \\
		\midrule
		\(\Omega(\tau)\) 
		& \(-0.030\) & \(0.075\) & \(0.180\) & \(0.090\) & \(-0.020\) & \(0.110\) & \(0.230\) & \(0.070\) & \(0.160\) \\
		\bottomrule
	\end{tabular}
\end{table}

In this case, the system remains mostly below the threshold
\[
\Theta=0.25,
\]
but it approaches it repeatedly. The regime is not collapse-like, because resilience and legitimacy are still sufficient to pull the system back. However, the oscillations indicate reduced robustness: the system no longer converges smoothly toward a stable equilibrium.

The mechanism can be interpreted as follows. Pressure rises, but resilience reacts late. When resilience finally increases, it reduces pressure, but the correction occurs after the system has already moved away from equilibrium. This lag generates alternating phases of stress accumulation and delayed correction.

\subsection{Scenario 3: Threshold crossing}

The third scenario describes delayed recognition of structural instability. The system starts below the threshold, remains apparently manageable for some time, and then crosses the instability threshold.

We increase the strength of the destabilizing interactions while reducing the effective stabilizing response:
\[
a_{CH}=0.12,\quad
a_{CP}=0.08,\quad
a_{HC}=0.10,\quad
a_{PC}=0.12,\quad
a_{PH}=0.12,
\]
and
\[
a_{CR}=0.04,\quad
a_{HL}=0.04,\quad
a_{HR}=0.03,\quad
a_{PR}=0.05.
\]
The common delay is
\[
\delta_{ij}=8.
\]

The representative values of \(\Omega(\tau)\) are shown in Table~\ref{tab:omega_threshold}.

\begin{table}[h!]
	\centering
	\caption{Scenario 3: threshold crossing.}
	\label{tab:omega_threshold}
	\begin{tabular}{cccccccccc}
		\toprule
		\(\tau\) & 0 & 10 & 20 & 30 & 40 & 50 & 60 & 80 & 100 \\
		\midrule
		\(\Omega(\tau)\) 
		& \(-0.030\) & \(0.072\) & \(0.151\) & \(0.228\) & \(0.311\) & \(0.410\) & \(0.532\) & \(0.889\) & \(1.498\) \\
		\bottomrule
	\end{tabular}
\end{table}

Since
\[
\Theta=0.25,
\]
the threshold is crossed between
\[
\tau=30
\quad\text{and}\quad
\tau=40.
\]
By linear interpolation, the estimated crossing time is
\[
\tau^\ast
\approx
30+
\frac{0.25-0.228}{0.311-0.228}\cdot 10
\approx
32.65.
\]

If recognition is delayed by an additional interpretive lag
\[
\Delta=6,
\]
then the system recognizes instability only at
\[
\tau^\ast+\Delta
\approx
38.65.
\]

This is the numerical analogue of delayed recognition:
\[
\Omega(\tau^\ast)\geq \Theta
\quad
\text{but}
\quad
\text{recognized instability occurs only at } \tau^\ast+\Delta.
\]

The important point is that the threshold crossing is not the beginning of structural change. It is the moment at which accumulated change becomes dynamically and interpretively effective. The idea is consistent with the general theory of delayed differential equations, where delays can transform gradual accumulation into sudden changes of qualitative behavior \cite{rihan2021dde,saeedian2022delay}.

\subsection{Scenario 4: Collapse-like transition}

The fourth scenario represents a collapse-like transition. The term ``collapse-like'' is used carefully: the model does not claim to reproduce historical collapse in full detail. Rather, it describes a mathematical regime in which destabilizing variables amplify each other while stabilizing variables decline.

The destabilizing interactions are further strengthened:
\[
a_{CH}=0.18,\quad
a_{CP}=0.16,\quad
a_{HC}=0.15,\quad
a_{PC}=0.18,\quad
a_{PH}=0.18.
\]
The stabilizing interactions are weakened:
\[
a_{CR}=0.025,\quad
a_{HL}=0.025,\quad
a_{HR}=0.020,\quad
a_{PR}=0.030.
\]
The delay is increased to
\[
\delta_{ij}=10.
\]

A representative numerical profile is given in Table~\ref{tab:omega_collapse}.

\begin{table}[h!]
	\centering
	\caption{Scenario 4: collapse-like transition.}
	\label{tab:omega_collapse}
	\begin{tabular}{cccccccccc}
		\toprule
		\(\tau\) & 0 & 10 & 20 & 30 & 40 & 50 & 60 & 80 & 100 \\
		\midrule
		\(\Omega(\tau)\) 
		& \(-0.030\) & \(0.451\) & \(1.209\) & \(2.774\) & \(4.647\) & \(4.988\) & \(4.999\) & \(5.000\) & \(5.000\) \\
		\bottomrule
	\end{tabular}
\end{table}

In this regime, the threshold is crossed very early:
\[
\Omega(10)=0.451>\Theta.
\]
The system rapidly exits the stable domain. Complexity, entropy, and pressure reinforce each other, while resilience and legitimacy are unable to compensate. In qualitative terms, the system loses its stabilizing attractor.

This regime should be interpreted as a mathematical warning. If stabilizing feedback is weak and delayed, the system may not merely approach instability; it may accelerate into it. Similar delay-induced transitions are well known in nonlinear delayed systems, where feedback lags can generate oscillations, bifurcations, or destabilization even when the corresponding non-delayed system remains stable \cite{rihan2021dde,saeedian2022delay,li2024fsdde}.

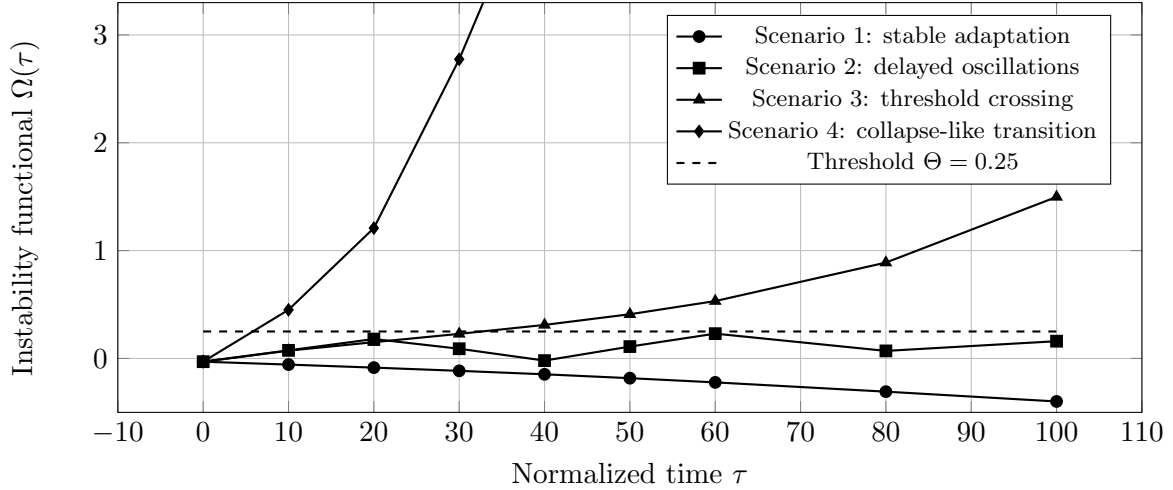
\begin{figure}[h!]
	\centering
	\begin{tikzpicture}
	\begin{axis}[
	width=0.95\textwidth,
	height=7cm,
	xlabel={Normalized time \(\tau\)},
	ylabel={Instability functional \(\Omega(\tau)\)},
	legend pos=north east,
	grid=both,
	ymin=-0.5,
	ymax=3.3
	]
	\addplot[thick, mark=*] coordinates {
		(0,-0.030) (10,-0.057) (20,-0.085) (30,-0.115) (40,-0.147)
		(50,-0.183) (60,-0.222) (80,-0.308) (100,-0.399)
	};
	\addlegendentry{\footnotesize Scenario 1: stable adaptation}
	
	\addplot[thick, mark=square*] coordinates {
		(0,-0.030) (10,0.075) (20,0.180) (30,0.090) (40,-0.020)
		(50,0.110) (60,0.230) (80,0.070) (100,0.160)
	};
	\addlegendentry{\footnotesize Scenario 2: delayed oscillations}
	
	\addplot[thick, mark=triangle*] coordinates {
		(0,-0.030) (10,0.072) (20,0.151) (30,0.228) (40,0.311)
		(50,0.410) (60,0.532) (80,0.889) (100,1.498)
	};
	\addlegendentry{\footnotesize Scenario 3: threshold crossing}
	
	\addplot[thick, mark=diamond*] coordinates {
		(0,-0.030) (10,0.451) (20,1.209) (30,2.774) (40,4.647)
		(50,4.988) (60,4.999) (80,5.000) (100,5.000)
	};
	\addlegendentry{\footnotesize Scenario 4: collapse-like transition}
	
	\addplot[dashed, thick] coordinates {(0,0.25) (100,0.25)};
	\addlegendentry{\footnotesize Threshold \(\Theta=0.25\)}
	\end{axis}
	\end{tikzpicture}
	\caption{Representative numerical profiles of the instability functional \(\Omega(\tau)\). Scenario 1 remains safely below the instability threshold. Scenario 2 displays delayed oscillations. Scenario 3 crosses the threshold after a period of gradual accumulation. Scenario 4 rapidly enters a collapse-like regime.}
	\label{fig:numerical_scenarios_omega}
\end{figure}

\subsection*{Interpretive synthesis}

The four scenarios can be summarized as follows.

\begin{table}[h!]
	\centering
	\caption{Qualitative interpretation of the numerical scenarios.}
	\label{tab:scenario_summary}
	\begin{tabular}{lll}
		\toprule
		Scenario & Mathematical behavior & Structural interpretation \\
		\midrule
		Stable adaptation & Convergence below threshold & Resilience absorbs stress \\
		Delayed oscillations & Oscillations near threshold & Stabilizing feedback acts late \\
		Threshold crossing & \(\Omega(\tau^\ast)\geq\Theta\) & Latent instability becomes visible \\
		Collapse-like transition & Rapid divergence above threshold & Stabilizing attractor is lost \\
		\bottomrule
	\end{tabular}
\end{table}

The main conclusion is that the delay parameters \(\delta_{ij}\) are not secondary. They determine whether stabilizing feedback arrives in time. Small delays may preserve stability; intermediate delays may generate oscillations; large delays, combined with strong destabilizing interactions, may generate threshold crossing or collapse-like acceleration.

From the perspective of delayed recognition, the numerical scenarios support the central thesis of the model:
\[
\text{structural instability may precede recognized instability.}
\]
A system may cross its instability threshold at \(\tau^\ast\), while recognition occurs only at \(\tau^\ast+\Delta\). During this interval, the system continues to behave as if the previous interpretive frame were valid, even though the underlying structural condition has already changed.





\section{Conclusion}

This paper proposed a delayed generalized Lotka--Volterra model for systems in which structural change and recognition of change are temporally separated. The model is based on the idea that instability may accumulate before it becomes visible.

The main mathematical object is the delayed system \eq{eq:general-delayed-lv}
combined with an instability functional \eq{eq:omega-functional}

The system remains structurally stable while
$
\Omega(\tau)<\Theta.
$
It crosses into structural instability when
$
\Omega(\tau)\geq\Theta.
$

The delayed terms express the fact that one variable may affect another only after cognitive, institutional, informational, or adaptive delay. The mathematical analysis shows that delay may destabilize equilibria, generate oscillations, and produce nonlinear transitions. Thus, delayed recognition is not merely an interpretive phenomenon; it can be modeled as a dynamical mechanism.

\bibliographystyle{plain}

\end{document}